\documentclass[11pt]{amsart}
\usepackage{geometry}                
\usepackage{graphicx}
\usepackage{amssymb}

\usepackage{todonotes}

\usepackage{amsthm}
\newtheorem{theorem}{Theorem}[section]
\newtheorem{proposition}[theorem]{Proposition}
\newtheorem{lemma}[theorem]{Lemma}

\newtheorem{remark}[theorem]{Remark}
\newtheorem{corollary}[theorem]{Corollary}

\newtheorem{question}[theorem]{Question}

\usepackage{epstopdf}
\title{Separability properties of automorphisms of graph products of groups}

\author{Michal Ferov}
\address{Building 54, Mathematical Sciences, University of Southampton, Highfield, Southampton SO17 1BJ, UK}
\email{michalferov@gmail.com}

\keywords{graph products, residual properties, pro-$\C$ topologies, pointwise-inner automorphisms, outer automorphisms.}
\subjclass[2010]{20E06, 20E26, 20E36, 20E45, 20F28, 20F65}

\let\star\relax
\let\S\relax
\let\P\relax

\newcommand{\bigast}{\mathop{\scalebox{1.5}{\raisebox{-0.2ex}{$\ast$}}}}

\DeclareMathOperator{\C}{\mathcal{C}}
\DeclareMathOperator{\NC}{\mathcal{N_C}}

\DeclareMathOperator{\Aut}{Aut}
\DeclareMathOperator{\Out}{Out}
\DeclareMathOperator{\Inn}{Inn}
\DeclareMathOperator{\End}{End}
\DeclareMathOperator{\Hom}{Hom}
\DeclareMathOperator{\Tor}{Tor}

\DeclareMathOperator{\Autpi}{Aut_{PI}}

\DeclareMathOperator{\fileq}{\leq_{\text{f.i.}}}
\DeclareMathOperator{\normfileq}{\unlhd_{\text{f.i.}}}
\DeclareMathOperator{\normleq}{\unlhd}

\DeclareMathOperator{\link}{link}
\DeclareMathOperator{\star}{star}
\DeclareMathOperator{\supp}{supp}

\DeclareMathOperator{\LL}{LL}
\DeclareMathOperator{\FL}{FL}
\DeclareMathOperator{\S}{S}
\DeclareMathOperator{\s}{s}
\DeclareMathOperator{\P}{P}
\DeclareMathOperator{\p}{p}

\DeclareMathOperator{\id}{id}

\DeclareMathOperator{\PT}{\mathcal{PT}}

\begin{document}

\begin{abstract}
	We study properties of automorphisms of graph products of groups. We show that graph product $\Gamma\mathcal{G}$ has non-trivial pointwise inner automorphisms if and only if some vertex group corresponding to a central vertex has non-trivial pointwise inner automorphisms. We use this result to study residual finiteness of $\Out(\Gamma \mathcal{G})$. We show that if all vertex groups are finitely generated residually finite and the vertex groups corresponding to central vertices satisfy certain technical (yet natural) condition, then $\Out(\Gamma\mathcal{G})$ is residually finite. Finally, we generalise this result to graph products of residually $p$-finite groups to show that if $\Gamma\mathcal{G}$ is a graph product of finitely generated residually $p$-finite groups such that the vertex groups corresponding to central vertices satisfy the $p$-version of the technical condition then $\Out(\Gamma\mathcal{G})$ is virtually residually $p$-finite. We use this result to prove bi-orderability of Torreli groups of some graph products of finitely generated residually torsion-free nilpotent groups. 
\end{abstract}

\setcounter{tocdepth}{1}
\maketitle
\tableofcontents

\section{Introduction and motivation}\label{Introduction}
\subsection{Motivation}We say that a $G$ is \emph{residually finite} (RF) if for every $g \in G \setminus \{1\}$ there is a finite group $F$ and a homomorphism $\varphi \colon G \to F$ such that $\varphi(g) \neq 1$ in $F$. The main motivation to study residually finite groups is that they can be approximated by their finite quotients. In case of finitely presented groups this approximation can be used to solve the word problem: Mal'cev \cite{malcev} constructed an algorithm that uniformly solves the word problem in the class of finitely presented RF groups. 

Baumslag \cite{baumslag} proved that if $G$ is a finitely generated RF group then $\Aut(G)$, the group of automorphisms of $G$, is RF as well. One could ask whether this result can be generalised to $\Out(G) \cong \Aut(G)/\Inn(G)$, the group of outer automorphisms? Negative answer to this question was provided by Bumagin and Wise in \cite{bumagina} when they proved that for every finitely presented group $O$ there is a finitely generated residually finite group $G$ such that $\Out(G) \cong O$. The question that naturally arises is: what properties does a finitely generated RF group $G$ need to satisfy to ensure that $\Out(G)$ is RF as well?

\subsection{Grossman's criterion}
An map $\phi: G \to G$ is \emph{pointwise inner} if $\phi(g)$ is conjugate to $g$ for every $g \in G$. Let $\Autpi(G)$ denote the set of all pointwise inner automorphisms of $G$. We say that a group $G$ has \emph{Grossman's property (A)} if every pointwise inner automorphism of $G$ is inner, i.e. if $\Autpi(G) = \Inn(G)$. We say that $G$ is \emph{conjugacy separable} (CS) if for every tuple $f,g \in G$ such that $f$ is not conjugate to $g$ there exists a finite group $F$ and a homomorphism $\varphi \colon G \to F$ such that $\varphi(f)$ is not conjugate to $\varphi(g)$. Grossman \cite[Theorem 1]{grossman} proved that if $G$ is a finitely generated CS group with Grossman's property (A) then $\Out(G)$ is residually finite. We call groups that satisfy this criterion \emph{Grossmanian groups}, i.e. group $G$ is Grossmanian if $G$ is finitely generated CS group with Grossman's property (A). However, these are not necessary conditions; see Section \ref{cias} for the discussion.

\subsection{Statement of results.}Another natural question to ask is how does the property of having a residually finite group of outer automorphisms behave under group theoretic constructions? In this paper we study the case of graph products of groups, which naturally generalise the notion of free products and direct products in the category of groups.

 Let $\Gamma$ be a simplicial graph, i.e. $V\Gamma$ is a set and $E\Gamma \subseteq \binom{V\Gamma}{2}$, and let $\mathcal{G} = \{G_v \mid v \in V\Gamma\}$ be a family of non-trivial groups. The group $\Gamma \mathcal{G}$, the \emph{graph product} of the family $\mathcal{G}$ with respect to graph $\Gamma$, is the quotient of the free product $\ast_{v \in V\Gamma}G_v$ modulo relations of the form
\begin{displaymath}
	 g_u g_v = g_v g_u \quad \forall g_u \in G_u, \forall g_v \in G_v \mbox{ whenever }\{u,v\}\in E\Gamma.
\end{displaymath}
The groups $G_v$ are called \emph{vertex groups}. In this study we will be considering only finite graph products, i.e. $|V\Gamma| < \infty$. Clearly, if $\Gamma$ is totally disconnected then the graph product $\Gamma\mathcal{G}$ is equal to $\ast_{v \in V\Gamma}G_v$, the free product of the vertex groups, and similarly, if $\Gamma$ is complete then $\Gamma\mathcal{G}$ is equal to $\times_{v \in V\Gamma}G_v$, the direct product of the vertex groups. Note that we will always assume that the vertex groups are non-trivial, i.e. $G_v \neq \{1\}$ for all $v \in V\Gamma$.

In the case when all vertex groups are infinite cyclic we are talking about \emph{right angled Artin groups} (RAAGs). If all vertex groups are cyclic of order two then we are talking about \emph{right angled Coxeter groups} (RACGs). We quickly list some partial results on residual finiteness of outer automorphisms of graph products of residually finite groups.
In \cite{raags} Minasyan showed that if $G$ is a finitely generated RAAG then $G$ is CS and has Grossman's property (A), hence $\Out(G)$ is RF by Grossmans criterion. Independently of Minasyan, Charney and Vogtmann \cite{vogtmann} proved that outer automorphism groups of finitely generated RAAGs are RF. Caprace and Minasyan \cite{caprace} proved that if $G$ is a finitely generated RACG then $\Out(G)$ is RF.

Following the results presented in \cite{osin} residual finiteness of outer automorphism groups of free products of finitely generated RF groups is well understood. However, it is not known whether the class of finitely generated RF groups with RF group of outer automorphisms is closed under direct products (see Question \ref{question 1}). To overcome this obstacle we introduce the following class of groups.

We say that a group $G$ is \emph{inner automorphism separable} (IAS) if every non-trivial outer automorphism of $G$ can be realised as a non-trivial outer automorphism of some finite quotient of $G$ (see Section \ref{cias} for formal definition). Obviously, if $G$ is IAS then $\Out(G)$ is RF. In Section \ref{cias} we show that the class of IAS groups is closed under taking direct products (see Corollary \ref{ias direct}). In Section \ref{examples} we give examples of IAS groups. In particular, we show that Grossmanian groups are IAS (see Corollary \ref{grossmanian}) and that virtually polycyclic groups are IAS (see Lemma \ref{polycyclic groups are IS}).

For $v \in V\Gamma$ we define $\star(u) = \{v \in V\Gamma \mid \{u,v\}\in E\Gamma\} \cup \{u\}$. We say that a subset $U \subseteq V\Gamma$ is \emph{coneless} if $\bigcap_{u \in U}\star(u) = \emptyset$, i.e. if the vertices in $U$ do not have a common neighbour.
In Section \ref{pointwise inner automorphisms} we study pointwise inner endomorphisms of graph products and we prove the following theorem.
\begin{theorem}
\label{no central vertex implies property A}
	Let $\Gamma$ be a graph and suppose that there is $U \subseteq V\Gamma$ such that $|U| < \infty$ and $U$ is coneless. Let $\mathcal{G} = \{G_v \mid v \in V\Gamma\}$ be a family of non-trivial groups and let $G = \Gamma \mathcal{G}$ be the graph product of $\mathcal{G}$ with respect to $\Gamma$. Then every pointwise inner endomorphism of $G$ is an inner automorphism. In particular, the group $\Gamma \mathcal{G}$ has Grossman's property (A).
\end{theorem}
In case when the graph $\Gamma$ is totally disconnected, i.e. when $\Gamma \mathcal{G} = \bigast_{v \in V\Gamma}G_v$ this has been proved by Minasyan and Osin in \cite{osin}.

We say that a graph $\Gamma$ is \emph{irreducible} if $V\Gamma$ cannot be expressed as a disjoint union of two non-empty subsets $V\Gamma = A \cup B$ such that $\{v_a, v_b\} \in E\Gamma$ for every $v_a \in A$ and $v_b \in B$. In case when the graph is irreducible then the result of Theorem \ref{no central vertex implies property A} can be recovered from \cite[Theorem 2.12]{acyl} and \cite[Theorem 7.5]{sisto}. 

We say that a vertex $v \in V\Gamma$ is \emph{central} in $\Gamma$ if $\{u,v\} \in E\Gamma$ for all $u \in V\Gamma \setminus \{v\}$, i.e. $v$ is central if it is adjacent to all the vertices of $\Gamma$ (apart from itself). It can be easily seen that in case when the graph $\Gamma$ is finite, then it satisfies the assumption of Theorem \ref{no central vertex implies property A} if and only if it does not contain central vertices. As a consequence of Theorem \ref{no central vertex implies property A} we give the following characterisation of finite graph products with Grossman's property (A).
\begin{corollary} 
	\label{theorem A}
	Let $\Gamma$ be a finite graph and let $\mathcal{G} = \{G_v \mid v \in V\Gamma\}$ be a family of non-trivial groups. The group $G = \Gamma \mathcal{G}$ has Grossman's property (A) if and only if all vertex groups corresponding to central vertices of $\Gamma$ have Grossman's property (A).
\end{corollary}

In Section \ref{CD pairs in GP} we study separability of conjugacy classes in graph products of residually finite groups and using Theorem \ref{no central vertex implies property A} we prove the following.
\begin{theorem}
	\label{corollary O}
	Let $\Gamma$ be a finite graph without central vertices and let $\mathcal{G} = \{G_v \mid v \in V\Gamma\}$ be a family of non-trivial finitely generated RF groups. Then the group $\Gamma \mathcal{G}$ is IAS and, consequently, $\Out(\Gamma \mathcal{G})$ is RF.
\end{theorem}
Note that this theorem generalises result of Minasyan and Osin in \cite{osin} on residual finiteness of outer automorphism groups of free products of finitely generated RF groups.

Combining Theorem \ref{corollary O} with Corollary \ref{ias direct} we obtain the following.
\begin{corollary}
	\label{theorem O}
	Let $\Gamma$ be a finite graph and let $\mathcal{G} = \{G_v \mid v \in V\Gamma\}$ be family of non-trivial finitely generated RF groups. Assume that $G_v$ is IAS whenever $v$ is central in $\Gamma$. Then the group $\Gamma\mathcal{G}$ is IAS and, consequently, $\Out(\Gamma\mathcal{G})$ is RF.
\end{corollary}

Next, combining Corollary \ref{theorem O} with Lemma \ref{polycyclic groups are IS} (virtually polycyclic groups are IAS) we get the following.
\begin{corollary}
	Let $\Gamma$ be a finite graph and let $\mathcal{G} = \{G_v \mid v \in V\Gamma\}$ be a family of non-trivial virtually polycyclic groups. Then the group $\Gamma \mathcal{G}$ is IAS and $\Out(\Gamma\mathcal{G})$ is RF.
\end{corollary}

Suppose that $p \in \mathbb{N}$ is a prime number. In Section \ref{section-p} we prove a $p$-analogue to Theorem \ref{corollary O}.
\begin{theorem}
	\label{virtually residually p}
	Let $\Gamma$ be a finite graph without central vertices and let $\mathcal{G} = \{G_v \mid v \in V\Gamma\}$ be a family of non-trivial finitely generated residually $p$-finite groups. Then the group $\Gamma \mathcal{G}$ is $p$-IAS, $\Out_p(\Gamma \mathcal{G})$ is residually $p$-finite and $\Out(\Gamma \mathcal{G})$ is virtually residually $p$-finite.
\end{theorem}
See Section \ref{section-p} for definitions of residually $p$-finite groups, $p$-IAS groups and $\Out_p$. As a matter of fact, we show that the class of finitely generated residually $p$-finite $p$-IAS groups is closed under direct products (See Lemma \ref{p-IAS direct}) and using that we prove a $p$-analogue to Corollary \ref{theorem O}.
\begin{corollary}
	\label{p-corollary}
	Let $\Gamma$ be a finite graph and let $\mathcal{G} = \{G_v \mid v\in V\Gamma\}$ be a family of non-trivial finitely generated residually $p$-finite groups. Assume that $G_v$ is $p$-IAS whenever $v$ is central in $\Gamma$. Then the group $\Gamma\mathcal{G}$ is $p$-IAS and, consequently, $\Out_p(\Gamma\mathcal{G})$ is residually $p$-finite and $\Out(\Gamma\mathcal{G})$ is virtually residually $p$-finite. 
\end{corollary}

Let $G$ be a group. The \emph{Torelli group} of $G$, $\Tor(G) \leq \Out(G)$, consists of all outer automorphisms of $G$ that act trivially on the abelianisation of $G$; see Section \ref{section-p} for formal definition of $\Tor(G)$. Finally, we use Theorem \ref{virtually residually p} to establish bi-orderability for Torelli groups of certain graph products of residually torsion-free nilpotent groups.
\begin{theorem}
\label{torreli}
	Let $\Gamma$ be a finite graph without central vertices $\Gamma$ and let $\mathcal{G} = \{G_v \mid v \in V\Gamma\}$ be a family of non-trivial residually torsion-free nilpotent groups. Then $\Tor(G)$ is residually $p$-finite for every prime number $p$ and is bi-orderable.
\end{theorem}

\section{Direct products of $\C$-IAS groups and Baumslag's method}\label{cias}
Let $G$ be a group and suppose that $H \leq G$; we will use $H \fileq G$ to denote that $|G:H| < \infty$. Similarly, we will use $N \normfileq G$ to denote that $N \normleq G$ and $|G:N|<\infty$.

\subsection{Pro-$\C$ topologies on groups}
Let $G$ be a group and let $\C$ be a class of finite groups. If $F \in \C$ then we say that $F$ is a \emph{$\C$-group}. We say that $N\normleq G$ is a \emph{co-$\C$} subgroup of $G$ if $G/N \in \mathcal{C}$ and we say that $G/N$ is a \emph{$\C$-quotient} of $G$. We will use $\NC(G) = \{N \normleq G \mid G/N \in \C\}$ to denote the set of co-$\C$ subgroups of $G$. In this paper we will always assume that the class $\C$ satisfies the following closure properties:
\begin{itemize}
	\item[(c1)] subgroups: let $G \in \C$ and $H \leq G$; then $H \in \C$, 
	\item[(c2)] finite direct products: let $G_1, G_2 \in \C$; then $G_1 \times G_2 \in \C$.
\end{itemize}
In this case one can easily check that for every group $G$ the system of subsets $\mathcal{B_C} = \{gN \mid g \in G, N \in \NC(G)\} \subseteq \mathcal{P}(G)$ forms a basis of open sets for a topology on $G$. This topology is called the \emph{pro-$\C$ topology} on $G$ and we will use pro-$\C(G)$ when referring to it. If $\C$ is the class of all finite groups then the corresponding group topology is called the \emph{profinite topology} on $G$ and is denoted $\mathcal{PT}(G)$. If $\C$ is the class of all finite $p$-groups, where $p$ is a prime number, then the corresponding group topology is referred to as \emph{pro-$p$} topology on $G$ and is denoted as pro-$p(G)$.

We say that a subset $X \subseteq G$ is $\C$-\emph{closed} or $\C$-\emph{separable} in $G$ if it is closed in pro-$\C(G)$; $\C$-open subsets of $G$ are defined analogically. One can show that if the class $\C$ satisfies (c1) and (c2) then, equipping a group with its pro-$\C$ topology, is actually a faithful functor from the category of groups to the category of topological groups, i.e. group homomorphisms are continuous with respect to corresponding pro-$\C$ topologies and group isomorphisms are homeomorphisms.

We say that a group $G$ is residually-$\C$ if for every $g \in G \setminus \{1\}$ there is $N \in \NC(G)$ such that $g \not\in N$. One can easily check that for a group $G$ the following are equivalent:
\begin{itemize}
	\item $G$ is residually-$\C$,
	\item $\{1\}$ is $\C$-closed in $G$,
	\item $\bigcap_{N \in \NC(G)}N = \{1\}$,
	\item pro-$\C(G)$ is Hausdorff.
\end{itemize}

\subsection{$\C$-IAS groups}

We say that a group $G$ is $\C$-\emph{inner automorphism separable} ($\C$-IAS) if for every $\phi \in \Aut(G)\setminus \Inn(G)$ there is $K \in \NC(G)$ characteristic in $G$ such that for the homomorphism $\tilde{\kappa} \colon \Aut(G) \to \Aut(G/K)$ given by
\begin{displaymath}
	\tilde{\kappa}(\psi)(gK) = \psi(g)K	
\end{displaymath}
for every $\psi \in \Aut(G)$ and $g \in G$
we have $\tilde{\kappa}(\phi) \not \in \Inn(G/K)$. In other words, group $G$ is $\C$-IAS if every outer automorphism of $G$ can be realised as non-trivial outer automorphism of some $\C$-quotient of $G$. If $\C$ is the class of all finite groups then we will say that $G$ is IAS. Similarly, if $\C$ is the class of all finite $p$-groups then we say that $G$ is $p$-IAS.

As stated in the introduction, Grossman proved that if $G$ is a finitely generated CS group with Grossman's property (A) (i.e. G is Grossmanian) then $\Out(G)$ is RF. However, these are not necessary conditions. Finite groups are trivially CS and the group of outer automorphisms of a finite group is finite, thus residually finite, yet examples of finite groups that do not have Grossman's property (A) were given by Burnside in \cite{burnside} and later by Sah in \cite{sah}.

Another example of redundancy of these conditions are virtually polycyclic groups. It was shown by Formanek in \cite{formanek} and independently by Remeslennikov in cite \cite{remeslennikov} that virtually polycyclic groups are CS, in \cite{segal} Segal gave a construction of a torsion-free polycyclic group with non-inner pointwise inner automorphisms, i.e. without Grossman's property (A), yet by a result by Werhfritz \cite{wehrfritz}, if $G$ is a virtually polycyclic group then $\Out(G)$ embeds in $\mathop{GL}_n(\mathbb{Z})$ for some $n \in \mathbb{Z}$ and thus is RF by Mal'cev's theorem (see \cite{linear}).

The above examples show that Grossman's property (A) is in not a necessary condition. Similarly, the group $\mathop{SL}_3(\mathbb{Z}) \ast \mathbb{Z}$ is RF but not CS, has Grossman's property (A) and $\Out(\mathop{SL}_3(\mathbb{Z})\ast \mathbb{Z})$ is RF (see \cite[Theorem 1.6]{osin}), thus conjugacy separability is not necessary either.

As stated in the introduction, it can be easily seen that if $G$ is IAS then $\Out(G)$ is RF. In Section \ref{section-p} we show that if $G$ is $p$-IAS then $\Out(G)$ is virtually residually $p$-finite (see Lemma \ref{p-CIAS virtually residually p}).
  
\subsection{Direct products of $\C$-IAS groups}

The following simple lemma demonstrates the idea that every subgroup of finite index contains a characteristic subgroup of a finite index.
\begin{lemma}
	\label{characteristic}
	Suppose that $\C$ is a class of finite groups satisfying (c1) and (c2) and let $G$ be finitely generated group Then for every $K \in \NC(G)$ there is $L \in \NC(G)$ such that $L \leq K$ and $L$ is characteristic in $G$. 
\end{lemma}
\begin{proof}
	Set
	\begin{displaymath}
		L = \bigcap_{\alpha \in \Aut(G)}\alpha^{-1}(K).
	\end{displaymath}
	Clearly, $L \leq K$ and it can be easily seen that $L$ is characteristic in $G$. Note that for every $\varphi \in \Aut(G)$ we have $\alpha^{-1}(K) \in \NC(G)$ and $|G:\alpha^{-1}(K)| = |G:K|$. Note that as $G$ is finitely generated there are only finitely many $H \leq G$ such that $|G:N| = n$ for every $n \in \mathbb{N}$. We see that $L$ is an intersection of finitely many co-$\C$ subgroups of $G$ and thus $L$ itself is a co-$\C$ subgroup of $G$.
\end{proof}

The main goal of this section is to adapt Grossman's method to prove the following.
\begin{proposition}
	\label{cias-direct}
	Let $\C$ be a class of finite groups satisfying (c1) and (c2). Let $A,B$ be finitely generated $\C$-IAS residually-$\C$ groups. Then the group $A\times B$ is $\C$-IAS.
\end{proposition}

For a group $G$ we will use $\End(G)$ to denote the set of all endomorphisms of $G$. Similarly, for groups $A,B$ we will use $\Hom(A,B)$ denote the set of all homomorphisms from $A$ to $B$. Note that $\Hom(A,A) = \End(A)$ for every group $A$.

Now let $A,B$ be groups and let $\phi \in \End(A \times B)$ be arbitrary. Set $\phi_A = \phi \restriction_{A \times \{1\}}$ and $\phi_B = \phi \restriction_{\{1\}\times B}$. Obviously, for $a \in A, b\in B$ we have $\phi((a,b)) = \phi_A((a,1))\phi_B((1,b))$. It is easy to see that there are uniquely given $\alpha \in \End(A)$ and $\gamma \in \Hom(A,B)$ such that $\phi_A((a,1)) = (\alpha(a),\gamma(a))$. Similarly, there are uniquely given $\delta \in \Hom(B,A)$ and $\beta \in \End((1,B))$ such that $\phi_B((1,b)) = (\delta(b),\beta(b))$. We sum up this simple observation in the following simple remark, which will be crucial for proving Proposition \ref{cias-direct}.
\begin{remark}
	\label{automorphism of a direct product}
	Let $A,B$ be groups. For every $\phi \in \End(A\times B)$ there are uniquely given $\alpha \in \End(A)$, $\gamma \in \Hom(A,B)$ and $\beta \in \End(B)$, $\delta \in \Hom(B,A)$ such that $\phi((a,b)) = (\alpha(a)\delta(b), \gamma(a)\beta(b))$ for all $a \in A$ and $b \in B$.
\end{remark}

Let $A, B$ be groups, suppose that $K_A \normleq A$, $K_B \normleq B$ and let $\psi_A \colon  A \to A/K_A$, $\psi_B \colon B \to B/K_B$ be the corresponding natural projections. Clearly, the map 
\begin{displaymath}
	\tilde{\psi}_{A,B} \colon \Hom(A,B) \to \Hom(A/K_A,B/K_B)
\end{displaymath}
given by 
\begin{displaymath}
	\tilde{\psi}_{A,B}(\phi)(a K_A) = \phi(a)K_B,
\end{displaymath}
for all $a \in A$ and $\phi \in \Hom(A,B)$ is well defined if and only if $\phi(K_A) \subseteq K_B$ for every $\phi \in \Hom(A,B)$, or equivalently, if $K_A \subseteq \phi^{-1}(K_B)$ for every $\phi \in \Hom(A,B)$. We use this observation together with Remark \ref{automorphism of a direct product} to show that the idea of the proof of Lemma \ref{characteristic} can be adapted to direct products in a way that the preserves the  structure of a direct product.

\begin{lemma}
\label{Baumslag direct}
	Let $A,B$ be finitely generated groups and let $K_A \in \NC(A)$, $K_B \in \NC(B)$ be arbitrary. Then there are $L_A \in \NC(A)$, $L_B \in \NC(B)$ such that all of the following hold:
	\begin{enumerate}
		\item $L_A \leq K_A$ and $L_B \leq K_B$,
		\item $L_A$ is fully characteristic in $A$,
		\item $L_A \subseteq \gamma^{-1}(L_B)$ for all $\gamma \in \Hom(A,B)$,
		\item $L_B$ is fully characteristic in $B$,
		\item $L_B \subseteq \delta^{-1}(L_A)$ for every $\delta \in \Hom(B,A)$
		\item $L_A \times L_B$ is fully characteristic in $A \times B$.
	\end{enumerate} 
\end{lemma}
\begin{proof}
	Set $k = \max\{|A:K_A|,|B:K_B|\}$ and denote
	\begin{displaymath} 
	\begin{split}
		\mathcal{L}_A = \{M \in \NC(A) \mid |A:M| \leq k\},\\
		\mathcal{L}_B = \{N \in \NC(B) \mid |B:N| \leq k\}.
	\end{split}
	\end{displaymath}	
	As $A$ is finitely generated, we see that for every $n \in \mathbb{N}$ there are only finitely many $H_A \leq A$ such that $|A:H_A| = n$, hence we see that $\mathcal{L}_A$ is a finite subset of $\NC(A)$. By a similar argument we see that $\mathcal{L}_B$ is a finite subset of $\NC(B)$. 
		Now set $L_A = \bigcap_{M \in \mathcal{L}_A}M$ and $L_B = \bigcap_{N \in \mathcal{L}_B}N$. As $\mathcal{L}_A$ is a finite subset of $\NC(A)$ we see that $L_A \in \NC(A)$ and by an analogous argument we see that $L_B \in \NC(B)$.
	
	Let $\alpha_0 \in \Hom(A,A)$ be arbitrary. Note that $\alpha_0^{-1}(M) \in \NC(A)$ and $|A:\alpha_0^{-1}(M)| \leq |A:M|$ for every $M \in \NC(A)$. Thus if $M \in \mathcal{L}_A$ then $\alpha_0^{-1}(M) \in \mathcal{L}_A$. We see that
	\begin{displaymath}
		\alpha_0^{-1}(L_A) 	= \alpha_0^{-1}\left( \bigcap_{M \in \mathcal{L}_A}M \right) = \bigcap_{M \in  \mathcal{L}_A} \alpha_0^{-1}(M) \supseteq \bigcap_{M\in \mathcal{L}_A} M = L_A.						
	\end{displaymath}
	and thus $L_A \subseteq \alpha_0^{-1}(L_A)$ for every $\alpha_0 \in \Hom(A,A)$, i.e. $L_A$ is fully characteristic in $A$.
	
	Similarly, for $\gamma_0 \in \Hom(A,B)$ we have $\gamma_0^{-1}(N) \in \NC(A)$ and $|A \colon \gamma_0^{-1}(N)| \leq |B:N|$ for every $N \in \NC(B)$ and thus if $N \in \mathcal{L}_B$ then $\gamma_0^{-1}(N) \in \mathcal{L}_A$. We see that
	\begin{displaymath}
		\gamma_0^{-1}(L_B) = \gamma_0^{-1}\left( \bigcap_{N \in \mathcal{L}_B} N \right) = \bigcap_{N \in \mathcal{L}_B} \gamma_0^{-1}(N) \supseteq \bigcap_{M \in \mathcal{L}_A}M = L_A
	\end{displaymath}
	and thus $L_A \subseteq \gamma_0^{-1}(L_B)$ for every $\gamma_0 \in \Hom(A,B)$.
	
	Using analogous arguments one can easily check that $L_B \subseteq \beta_0^{-1}(L_B)$ for every $\beta_0 \in \Hom(B,B)$, i.e. $L_B$ is fully characteristic in $B$, and $L_B \subseteq \delta_0^{-1}(L_A)$ for every $\delta_0 \in \Hom(B,A)$.
	
	Now, let $\phi \in \End(A \times B)$ be arbitrary. Following Remark \ref{automorphism of a direct product} we see that there are uniquely given $\alpha \in \Hom(A,A)$, $\beta \in \Hom(B,B)$, $\gamma \in \Hom(A,B)$ and $\delta \in \Hom(A,B)$ such that $\phi((a,b)) = (\alpha(a)\delta(b),\beta(b)\gamma(a))$ for all $a \in A$, $b \in B$. Note that $\alpha(L_A) \leq L_A$, $\beta(L_B) \leq L_B$, $\gamma(L_A) \leq L_B$ and $\delta(L_B) \leq L_A$. We see that
	\begin{displaymath}
		\phi(L_A \times L_B) \subseteq \alpha(L_A)\delta(L_B) \times \beta(L_B)\gamma(L_A) \subseteq L_A \times L_B
	\end{displaymath}
	and hence $L_A \times L_B$ is fully characteristic in $A \times B$.
\end{proof}

We say that a homomorphism $\alpha \colon A \to B$, where $A,B$ are groups, is \emph{trivial} if $\alpha(a) = 1$ for all $a \in A$. Before we proceed to the proof of Proposition \ref{cias-direct} we state one simple observation.
\begin{remark}
	\label{stable factors}
	Let $A,B$ be groups and let $\phi \in \Aut(A \times B)$. If $\phi \in \Inn(A \times B)$ then $\phi(A \times \{1\}) \subseteq A \times \{1\}$ and $\phi(\{1\} \times B) \subseteq \{1\}\times B$.
\end{remark}

Now we are ready to prove Proposition \ref{cias-direct}
\begin{proof}
	Let $\phi \in \Aut(A \times B)$ be arbitrary such that $\phi \not\in \Inn(A \times B)$. Following Remark \ref{stable factors} we see that there are two disjoint cases:
	\begin{itemize}
		\item[(i)] either $\phi(A \times \{1\}) \not \subseteq A \times \{1\}$ or $\phi(\{1\} \times B) \not \subseteq \{1\} \times B$,
		\item[(ii)] $\phi(A \times \{1\}) \subseteq A \times \{1\}$ and $\phi(\{1\} \times B) \subseteq \{1\} \times B$.
	\end{itemize}
	Following Remark \ref{automorphism of a direct product} we see that there are $\alpha \in \End(A)$, $\delta \in \Hom(B,A)$, $\gamma \in \Hom(A,B)$ and $\beta \in \End(B)$ such that $\phi(a,b) = (\alpha(a)\delta(b),\gamma(a)\beta(b))$ for all $a \in A$, $b \in B$.
		
	Suppose that (i) is the case. This means that either $\gamma$ is non-trivial or $\delta$ is non-trivial. Without loss of generality we may assume that $\delta$ is non-trivial, i.e. there is $b_0 \in B \setminus \{1\}$ such that $\delta(b_0) \in A \setminus \{1\}$. As both $A,B$ are residually-$\C$ there are $K_A \in \NC(A)$ and $K_B \in \NC(B)$ such that $\delta(b_0) \not\in K_A$ and $b_0 \not\in K_B$. By Lemma \ref{Baumslag direct} we see that there are $L_A \in \NC(A)$ and $L_B \in \NC(B)$ such that $L_A \leq K_A$, $L_B \leq K_B$, $L_A$ is fully characteristic in $A$, $L_B$ is fully characteristic in $B$, $L_A \subseteq \gamma^{-1}(L_B)$ for all $\gamma \in \Hom(A,B)$, $L_B \in \delta^{-1}(L_A)$ for every $\delta \in \Hom(B,A)$ and $L_A \times L_B$ is fully characteristic in $A \times B$. We see that the natural projections $\psi_A \colon A \to A/L_A$, $\psi_B \colon B \to B/L_B$ induce maps 
	\begin{displaymath}
	\begin{split}
		\tilde{\psi}_A \colon 		\Hom(A,A) \to \Hom(A/L_A, A/L_A),\\
		\tilde{\psi}_B \colon 		\Hom(B,B) \to \Hom(B/L_B,B/L_B),\\
		\tilde{\psi}_{A,B} \colon 	\Hom(A,B) \to \Hom(A/L_A, B/L_B),\\
		\tilde{\psi}_{B,A} \colon 	\Hom(B,A) \to \Hom(B/L_B, A/L_A).
	\end{split}
	\end{displaymath}
	Let $\psi \colon A\times B \to (A \times B)/(L_A \times L_B) = A/L_A \times B/B_L$ be the natural projection. Clearly, for the induced homomorphism $\tilde{\psi}\colon \Aut(A\times B) \to \Aut(A/L_A \times B/L_B)$ we have
	\begin{align*}
		\tilde{\psi}(\phi)(aL_A,bL_B) &= \left(\tilde{\psi}_A(\alpha)(aL_A)\tilde{\psi}_{B,A}(\delta)(b L_B), \tilde{\psi}_{A,B}(\gamma)(a L_A)\tilde{\psi}_B(\beta)(bL_B)\right)\\
		& =\left(\alpha(a)\delta(b)L_A, \gamma(a)\beta(b)L_B\right)
	\end{align*}
	for all $a\in A$, $b \in B$. Note that $b_0 \not \in L_B$ and $\delta(b_0) \not\in L_A$, thus $b_0 L_B$ is not the identity in $B/L_B$ and $\delta(b_0)L_A$ is not the identity in $A/L_A$. As $\tilde{\psi}_{B,A}(\delta)(b_0 L_B) = \delta(b_0)L_A$ we see that $\tilde{\psi}_{B,A}(\delta)$ is not trivial and consequently $\tilde{\psi}(\phi)(\{1\} \times B/L_B) \not \subseteq \{1\} \times B/L_B$. This means that $\tilde{\psi}(\phi) \not \in \Inn(A/L_A \times B/L_B)$.
	
	Suppose that (ii) is the case. This means that $\delta$, $\gamma$ are trivial and either $\alpha \in \Aut(A) \setminus \Inn(A)$ or $\beta \in \Aut(B) \setminus \Inn(B)$. Without loss of generality we may assume that $\alpha \not\in \Inn(A)$. Since $A$ is $\C$-IAS by assumption we see that there is $K_A \in \NC(A)$ characteristic in $A$ such that for the induced homomorphism $\tilde{\kappa}\colon \Aut(A) \to \Aut(A/L_A)$ we have $\tilde{\kappa}(\alpha) \not\in \Inn(A/K_A)$. Note that $B \in \NC(B)$, thus we can set $K_B = B$ and use Lemma \ref{Baumslag direct} to obtain $L_A \in \NC(A)$ and $L_B \in \NC(B)$ with the desired properties. We see that for the homomorphism $\tilde{\psi} \colon \Aut(A\times B) \to \Aut(A/L_A \times B/L_B)$ induced by the natural projection $\psi \colon A\times B \to A/L_A \times B/L_B$ we have $\tilde{\psi}(\phi)(a L_A , b L_B) = (\alpha(a) L_A, \beta(b)L_B)$ for all $a \in A$, $b \in B$. Since $L_A \leq K_A$ we see that $\tilde{\psi}_A(\alpha) \not \in \Inn(A/L_A)$ and thus $\tilde{\psi}(\phi) \not\in \Inn(A/L_A \times B/L_B)$.
	
	In each case we were able to realise the automorphism $\phi \in \Aut(A \times B) \setminus \Inn(A \times B)$ as non-inner automorphism of a $\C$-quotient of $A\times B$ and hence we see that the group $A\times B$ is $\C$-IAS.
\end{proof} 

Applying Proposition \ref{cias-direct} to the class of all finite groups we get the following corollary.
\begin{corollary}
	\label{ias direct}
	Let $A,B$ be finitely generated groups RF groups and suppose that both $A$ and $B$ are IAS. Then $A \times B$ is IAS and, consequently, $\Out(A \times B)$ is RF.
\end{corollary}

\section{Examples of $\C$-IAS groups and Grossmans method}\label{cias_examples}
\label{examples}
Let $G$ be a group and suppose that $H \leq G$. For $g \in G$ we will use $g^H$ to denote $\{hgh^{-1} \mid h \in H\}$, the $H$-conjugacy class of $g$. For $f,g \in G$ we will use $f \sim_H g$ to denote that $f \in g^H$. Suppose that $f \not\sim_G g$. We say that the pair $(f,g)$ is \emph{$\C$-conjugacy distinguishable} ($\C$-CD) in $G$ if there is a group $F \in \C$ and a homomorphism $\phi \colon G \to F$ such that $\phi(f) \not\sim_F \phi(g)$. Equivalently, the pair $(f,g)$ is $\C$-CD in $G$ if there is $N \in \NC(G)$ such that $f^G N \cap gN = \emptyset$ in $G$. We say that a group is \emph{$\C$-conjugacy separable} ($\C$-CS) if for every $f,g \in G$ such that $f\not\sim_G g$ there is a group $F \in \C$ and a homomorphism $\phi \colon G \to F$ such that $\phi(f) \not\sim_F \phi(g)$. Clearly, a group $G$ is $\C$-CS if for all $f,g \in G$ the pair $(f,g)$ is $\C$-CD whenever $f \not\sim_G g$. It is easy to see that the conjugacy class $g^G$ is $\C$-separable in $G$ if the pair $(f,g)$ is $\C$-CD for every $f \in G \setminus g^G$.

To simplify our proofs we will often use the following remark.
\begin{remark}
\label{CCD simplification}
	Let $G$ be a group and let $f,g \in G$ such that $f \not\sim_G g$. The pair $(f,g)$ is $\C$-CD if and only if there is a group $F$ and a homomorphism $\phi \colon G \to F$ such that $\phi(f) \not\sim_F \phi(g)$ and the pair $(\phi(f),\phi(g))$ is $\C$-CD in $F$.
\end{remark}

The following lemma uses an adaptation of the method that Grossman used to prove \cite[Theorem 1]{grossman}.
\begin{lemma}[Generalised Grossman's criterion]
	\label{my Grossman}
	Let $G$ be a finitely generated group and assume that for every $\phi \in \Aut(G) \setminus \Inn(G)$ there is an element $g \in G$ such that $\phi(g) \not\sim_G g$ and the pair $(\phi(g),g)$ is $\C$-CD in $G$. Then the group $G$ is $\C$-IAS.
\end{lemma}
\begin{proof}
	Take any $\phi \in \Aut(G) \setminus \Inn(G)$. By assumption, there is $g \in G$ such that $\phi(g) \not\sim_G g$ and the pair $(\phi(g),g)$ is $\C$-CD. There is $N \in \NC(G)$ such that $\phi(g)N \cap g^G N= \emptyset$. Set
	\begin{displaymath}
		K = \bigcap_{\varphi \in \Aut(G)} \varphi^{-1}(N).
	\end{displaymath}
	Obviously, $K$ is characteristic in $G$. Also, $|G : \varphi^{-1}(N)| \leq |G:N|$ for every $\varphi \in \Aut(G)$. As $G$ is finitely generated we see that $K$ is actually an intersection of finitely many co-$\C$ subgroups of $G$ and thus $K \in \NC(G)$. Let $\kappa \colon G \to G/K$ be the natural projection and let $\tilde{\kappa} \colon \Aut(G) \to \Aut(G/K)$ be the induced homomorphism. As $K \leq N$ we see that $\phi(g)K \cap g^G K = \emptyset$. This means that $\tilde{\kappa}(\phi)(gK) = \phi(g)K \not\sim_{G/K} gK$ and thus $\tilde{\kappa}(\phi) \not\in \Inn(G/K)$. We see that $G$ is $\C$-IAS.
\end{proof}

We say that a group $G$ is \emph{$\C$-Grossmanian} if $G$ is finitely generated, $\C$-CS group with Grossman's property (A).

\begin{corollary}
\label{grossmanian}
	If $G$ is a $\C$-Grossmanian group then $G$ is $\C$-IAS.
\end{corollary}
\begin{proof}
	Let $\phi \in \Aut(G) \setminus \Inn(G)$ be arbitrary. As $G$ has Grossman's property (A) we see that there is $g \in G$ such that $\phi(g) \not\sim_G g$. As $G$ is $\C$-CS we see that the pair $(\phi(g),g)$ is $\C$-CD in $G$. The group $G$ is $\C$-IAS by Lemma \ref{my Grossman}.
\end{proof}
As mentioned in the introduction, applying Corollary \ref{grossmanian} to the class of all finite groups we see that Grossmanian groups are IAS.

We say that a group $G$ satisfies the \emph{centraliser condition} (CC) if for every $g \in G$ and $K \normfileq G$ there is $L \normfileq G$ such that $L \leq K$ and
\begin{displaymath}
	C_{G/L}(\psi(g)) \leq \psi\left( C_G(g)K\right) \mbox{ in }G/L
\end{displaymath} 
where $\psi \colon G \to G/L$ is the natural projection.

We say that a group $G$ is \emph{hereditarily conjugacy separable} (HCS) if $G$ is CS and for every $H \fileq G$ we have that $H$ is CS as well. The following theorem was proved by Minasyan in \cite[Proposition 3.2]{raags}.
\begin{theorem}
\label{hcs and cc}
	Let $G$ be a group. Then the following are equivalent:
	\begin{itemize}
		\item[(a)] $G$ is HCS;
		\item[(b)] $G$ is CS and satisfies CC.
	\end{itemize}
\end{theorem}
Recall that a group is IAS if it is $\C$-IAS in the case when $\C$ is the class of all finite groups. Before we proceed to utilise Minasyan's theorem to show that virtually polycyclic groups are IAS we will need one more definition: we say that a group $G$ is \emph{double coset separable} if for every pair of finitely generated subgroups $H,K \leq G$ and an arbitrary element $g \in G$ the subset $HgK = \{hgk \mid h\in H, k \in K\}$ is separable in $\PT(G)$. Virtually polycyclic groups are double coset separable by \cite{polycyclic_DCS} and conjugacy separable by \cite{formanek, remeslennikov}. As every subgroup of a virtually polycyclic group is a virtually polycyclic group we see that virtually polycyclic groups are actually HCS.
\begin{lemma}
	\label{polycyclic groups are IS}
	Virtually polycyclic groups are IAS.
\end{lemma}
\begin{proof}
	Let $G$ be a virtually polycyclic group and let $\phi \in \Aut(G) \setminus \Inn(G)$ be arbitrary.
	
	There are two cases to be considered:
	\begin{itemize}
		\item[(i)] $\phi$ is not pointwise inner,
		\item[(ii)] $\phi$ is pointwise inner.
	\end{itemize}
	
	If (i) is the case then there is $g \in G$ such that $\phi(g) \not\sim g$. As stated before, virtually polycyclic groups are CS, thus there is $N \normfileq G$ such that $\phi(g)N \cap g^G N = \emptyset$. Following Lemma \ref{characteristic} we see that there is $N' \fileq G$ such that $N' \leq N$ and $N'$ is characteristic in $G$. Hence, if necessary, by replacing $N$ by $N'$ we might without loss of generality assume that $N$ is actually characteristic in $G$. Let $\nu \colon G \to G/N$ be the natural projection. Using the same argument as in the proof of Lemma \ref{my Grossman} we see that for the induced homomorphism $\tilde{\nu} \colon \Aut(G) \to \Aut(G/N)$ we have $\tilde{\nu}(\phi) \in \Aut(G/N)\setminus\Inn(G/N)$.
	
	Now suppose that (ii) is the case. Let $\{g_1, \dots, g_n \}\subseteq G$ be some generating set for $G$. By assumption for every $i\in \{1, \dots, n\}$ there is $c_i \in G$ such that $\phi(g_i) = c_i g_i c_i^{-1}$. Clearly there is no $c \in G$ such that $c g_i c^{-1} = c_i g_i c_i^{-1}$ for all $i = 1,\dots,n$ because otherwise the automorphism $\phi$ would be inner. Equivalently, $\phi$ is not inner if and only if 
	\begin{equation}
		\label{condition}
		c_1 C_G(g_1) \cap \dots \cap c_n C_G(g_n) = \emptyset \mbox{ in }G.
	\end{equation}
	Set $\overline{G} = G^n$, where $G^n$ is the $n$-fold direct product of $G$, and let $\overline{D} = \{(g,\dots,g) \mid g \in G\} \leq \overline{G}$ be the diagonal subgroup of $\overline{G}$. Clearly, the condition (\ref{condition}) holds if and only if $\overline{c} \not\in C_{\overline{G}}(\overline{g})\overline{D}$ in $\overline{G}$, where $\overline{g} = (g_1, \dots, g_n)\in \overline{G}$ and $\overline{c} = (c_1, \dots, c_n) \in \overline{G}$. Note that $\overline{G}$ is a virtually polycyclic group and thus it is double coset separable.  Every subgroup of a virtually polycyclic subgroup is virtually polycyclic and thus it finitely generated. Hence $C_{\overline{G}}(\overline{g}) \leq \overline{G}$ is finitely generated. By double coset separability of $\overline{G}$ we see that there is $N \normleq \overline{G}$ such that $|\overline{G}:N| < \infty$ and $\overline{c}N \cap  C_{\overline{G}}(\overline{g})\overline{D} = \emptyset$. Let $\iota_j: G \to \overline{G}$ be the injection of $G$ onto the $j$-th coordinate group of $\overline{G}$ for $j = 1, \dots, n$ and set 
	\begin{displaymath}
		K = \iota_1^{-1}\left(\iota_1(G)\cap N\right) \cap \dots \cap \iota_n^{-1}\left(\iota_n(G)\cap N\right) \leq G.
	\end{displaymath}
	Let $\overline{K} =  K^n \leq \overline{G}$ be the $n$-fold direct product of $K$. Note that $K \normfileq G$, $\overline{K} \normfileq \overline{G}$ and $\overline{K} \leq N$ thus $\overline{c}\overline{K} \cap C_{\overline{G}}(\overline{g}) \overline{D} = \emptyset$. 
This is equivalent to  
	\begin{displaymath}
		c_1C_G(g_1)K \cap \dots \cap c_n C_G(g_n)K = \emptyset \mbox{ in }G.
	\end{displaymath}
	Virtually polycyclic groups are hereditarily conjugacy separable and thus, by Theorem \ref{hcs and cc}, they satisfy CC. We see that for every $i \in \{1,\dots, n\}$ there is $L_i \normfileq G$ such that $L_i \leq K$ and
	\begin{displaymath}
		C_{G/L_i}(\psi_i(g_i)) \subseteq \psi_i(C_G(g_i)K) \mbox{ in } G/L_i,
	\end{displaymath}
where $\psi_i \colon G \to G/L_i$ is the natural projection. As $G$ is finitely generated, using the same argument as in (i), we may without loss of generality assume that $L_i$ is actually characteristic in $G$. Set $L = L_1 \cap \dots \cap L_n$. Clearly, $L$ is an intersection of finitely many characteristic subgroups of $G$ and therefore $L$ itself is characteristic in $G$. Note that $L \leq K$ and for every $i \in \{1, \dots, n\}$ we have
\begin{displaymath}
		C_{G/L}(\psi(g_i)) \subseteq \psi(C_G(g_i)K) \mbox{ in } G/L,
\end{displaymath}
	where $\psi \colon G \to G/L$ is the natural projection. We see that
	\begin{displaymath}
		\psi(c_1)C_{G/L}(\psi(g_1)) \cap \dots \cap \psi(c_n)C_{G/L}(\psi(g_n)) \subseteq \psi(c_1 C_{G}(g_1)K) \cap \dots \cap\psi(c_nC_G(g_n)K) \mbox{ in }G/L.
	\end{displaymath}
	Suppose that there is some $c \in G$ such that 
	\begin{displaymath}
		cL \in \psi(c_1 C_{G}(g_1)K) \cap \dots \cap\psi(c_nC_G(g_n)K) \mbox{ in }G/L.
	\end{displaymath}
	This means that
	\begin{displaymath}
		c \in \psi^{-1}\left(\psi(c_1 C_G(g_1)K) \cap \dots \cap \psi(c_n C_G(g_n)K)\right)	\subseteq c_1 C_G(g_1)K \cap \dots \cap c_n C_G(g_n)K = \emptyset
	\end{displaymath}
	which is a contradiction. Therefore
	\begin{displaymath}
		\psi(c_1)C_{G/L}(\psi(g_1)) \cap \dots \cap \psi(c_n)C_{G/L}(\psi(g_n)) = \emptyset \mbox{ in }G/L.
	\end{displaymath}
	It follows that for the induced homomorphism $\tilde{\psi} \colon \Aut(G) \to \Aut(G/L)$ we have $\tilde{\psi}(\phi) \not\in \Inn(G/L)$. We see that $G$ is IAS. 
\end{proof}

\section{Properties of graph products of groups}\label{graph products}

In this section we will recall some basic theory of graph products that was introduced in \cite{green} by Green and theory of cyclically reduced 
elements leading to conjugacy criterion for graph products of groups introduced in \cite{mf}.

Let $G = \Gamma \mathcal{G}$ be a graph product. Every $g \in G$ can be obtained as a product $g = g_1 \dots g_n$, where $g_i \in G_{v_i}$ for some $v_i \in V\Gamma$. However, this is not given uniquely. We say that a finite sequence $W \equiv (g_1, \dots, g_n)$ is a \emph{word} in $\Gamma \mathcal{G}$ if $g_i \in G_{v_i}$ for some $v_i \in V\Gamma$ for $i = 1, \dots, n$. We say that $g_i$ is a \emph{syllable} of $W$ and that the number $n$ is the \emph{length} of $W$. We say that the word $W$ \emph{represents} $g \in G$ if $g = g_1 \dots g_n$. We can define the following three types of transformations on the word $W$:
\begin{itemize}
\item[(T1)] remove a syllable $g_i$ if $g_i = 1$,
\item[(T2)] remove two consecutive syllables $g_i, g_{i+1}$ belonging to the same vertex group and replace them by a single syllable $g_i g_{i+1}$,
\item[(T3)] interchange consecutive syllables $g_i \in G_u$ and $g_{i+1} \in G_v$ if $\{u,v\}\in E\Gamma$. 
\end{itemize}
Transformations of type (T3) are called \emph{syllable shuffling}. Note that the transformations of types (T1) and (T2) reduce the length  of $W$ by 1, whereas (T3) preserves it. We say that word $W$ is \emph{reduced} if it is of minimal length, i.e. no sequence of transformations (T1) - (T3) will produce a word of shorter length. Obviously, if we start with a word $W$ representing an element $g \in G$ then by applying finitely many of the above transformations we will rewrite $W$ to a reduced word $W'$ that represents the same element $g$. The following theorem was proved by Green \cite[Theorem 3.9]{green} in her Ph.D. thesis.
\begin{theorem}[The normal form theorem]
    \label{normal form theorem for graph products}
    Every element $g \in \Gamma \mathcal{G}$ can be represented by a reduced word. Moreover, if two reduced words represent the same element of the group, then one can be obtained from the other by applying a finite sequence of syllable shuffling. In particular, the length of a reduced word is minimal among all words representing g, and a reduced word represents the identity if and only if it is the empty word.
\end{theorem}

Thanks to Theorem \ref{normal form theorem for graph products} the following definitions make sense. Let $g$ be an arbitrary element of $G$ and let $W \equiv (g_1, \dots, g_n)$ be a reduced word representing $g$ in $G$. We use $|g|=n$ to denote the \emph{length} of $g$ and we define the \emph{support} of $g$ to be 
\begin{displaymath}
	\supp(g) = \{v \in V\Gamma \mid \exists i \in \{1,\dots, n\} \mbox{ such that } g_i \in G_v\}.
\end{displaymath}
We define $\FL(g) \subseteq V\Gamma$ as the set of all $v \in V\Gamma$ such that there is a reduced word $W$ that represents the element $g$ and starts with a syllable from $G_v$. Similarly we define $\LL(g) \subseteq V\Gamma$ as the set of all $v \in V\Gamma$ such that there is a reduced word $W$ that represents the element $g$ and ends with a syllable from $G_v$. Note that $\FL(g) = \LL(g^{-1})$.

Every subset of vertices $X \subseteq V\Gamma$ induces a full subgraph $\Gamma_X$ of the graph $\Gamma$. Let $G_X$ be the subgroup of $G$ generated by the vertex groups corresponding to the vertices contained in $X$. Subgroups of $G$ that can be obtained in such way are called \emph{full subgroups} of $G$; according to standard convention, $G_{\emptyset} = \{1\}$. Using the normal form theorem one can easily show that $G_X$ is naturally isomorphic to the graph product of the family $\mathcal{G}_X = \{G_v \mid v \in X\}$ with respect to the full subgraph $\Gamma_X$. It is also easy to see that there is a canonical retraction $\rho_X \colon G \to G_X$ defined on the standard generators of $G$ as follows:
\begin{displaymath}
	\rho_X(g) = \begin{cases}
					g &\mbox{ if }g \in G_v \mbox{ for some } v \in X,\\
					1 &\mbox{ otherwise.} 
				\end{cases}
\end{displaymath}
We will often abuse the notation and sometimes consider the retraction $\rho_X$ as a surjective homomorphism $\rho_X \colon G \to G_X$ and sometimes as an endomorphism $\rho_X \colon G \to G$. In that case writing $\rho_{X} \circ \rho_{Y}$, where $Y \subseteq V\Gamma$, makes sense.

Let $A,B \subseteq V\Gamma$ be arbitrary. Let $G_A, G_B \leq G$ be the corresponding full subgroups of $G$ and let $\rho_A, \rho_B$ be the corresponding retractions. One can easily check that $\rho_A$ and $\rho_B$ commute: $\rho_A \circ \rho_B = \rho_B \circ \rho_A$. It follows that $G_A \cap G_B = G_{A \cap B}$ and $\rho_A \circ \rho_B = \rho_{A \cap B}$.

For a vertex $v \in V\Gamma$ we will use $\link(v)$ to denote $\{u \in V\Gamma \mid \{u,v\}\in E\Gamma\}$, the set of vertices adjacent to $v$, and we will use $\star(v)$ to denote $\link(v) \cup \{v\}$. If $S \subseteq V\Gamma$ then $\link(S) = \cap_{v \in S}\link(v)$ and $\star(S) = \cap_{v \in S}\star(v)$. 

We will use $N_G(H)$ to denote the normaliser of a subgroup $H$ in a group $G$. The following remark is a special case of \cite[Proposition 3.13]{yago}.
\begin{remark}
	\label{normaliser}
	Let $v \in V\Gamma$. Then $N_G(G_v) = G_{\star(v)} = G_v G_{\link(v)} \simeq G_v \times G_{\link(v)}$.	
\end{remark}

For $g \in G$ and $H \leq G$ we will use $C_H(g)$ to denote $\{c \in H \mid cg = gc\}$, the $H$-centraliser of $g$ in $G$. The following remark is a special case of \cite[Lemma 3.7]{mf}.
\begin{remark}
\label{centralisers_in_graph_products}
	Let $v \in V\Gamma$ and let $a \in G_v \setminus \{1\}$ be arbitrary. Then $C_G(a) = C_{G_v}(a)G_{\link(A)} \simeq C_{G_v}(a) \times G_{\link(A)}$.
\end{remark}

	Let $G = \Gamma \mathcal{G}$ be a graph product and let $g_1, \dots, g_n \in G$ be arbitrary. We say that the element $g = g_1 \dots g_n$ is a \emph{reduced product} of $g_1, \dots, g_n$ if $|g| = |g_1|+ \dots + |g_n|$.

	Let $g \in G$. We define $\S(g) = \supp(g) \cap \star(\supp(g))$. We also define $\P(g) = \supp(g) \setminus \S(g)$. Obviously $g$ uniquely factorises as a reduced product $g = \p(g) \s(g)$ where $\supp(\p(g)) =\P(g)$ and $\supp(\s(g)) = \S(g)$ . We call this factorisation the \emph{P-S decomposition of $g$}.

Let $g \in G$, let $W \equiv (g_1, \dots ,g_n)$ be a reduced expression for $g$. We say that a sequence $W' = (g_{j+1}, \dots, g_n, g_{1}, \dots, g_j)$, where $j \in \{1,\dots, n-1\}$, is a \emph{cyclic permutation} of $W$. We say that the element $g'\in G$ is a \emph{cyclic permutation} of $g$ if $g'$ can be expressed by a cyclic permutation of some reduced expression for $g$.
 
Let $W \equiv (g_1, \dots, g_n)$ be some reduced expression in $G$. We say that $W$ is \emph{cyclically reduced} if all cyclic permutations of $W$ are reduced. The following lemma was proved in \cite[Lemma 3.8]{mf}.
\begin{lemma}
	Let $g \in G$ be arbitrary and let $W \equiv (g_1, \dots g_n)$ be some reduced expression for $g$. If $W$ is cyclically reduced then all reduced expressions representing $g$ are cyclically reduced.
\end{lemma}

	Let $g \in G$ be arbitrary. We say that $g$ is \emph{cyclically reduced} if either $g$ is trivial or some reduced word representing $g$ is cyclically reduced. The following characterisation of cyclically reduced elements was given in \cite[Lemma 3.11]{mf}
\begin{lemma}
	\label{graph_product_cyclically_reduced_criterion}
	Let $g \in G$. Then the following are equivalent:
	\begin{enumerate}
		\item[(i)]	$g$ is cyclically reduced,
		\item[(ii)]	$(\FL(g)\cap \LL(g)) \setminus \S(g) = \emptyset$,
		\item[(iii)]$\FL(\p(g)) \cap \LL(\p(g)) = \emptyset$,
		\item[(iv)] $\p(g)$ is cyclically reduced.
	\end{enumerate}
\end{lemma}
One of the consequences of Lemma \ref{graph_product_cyclically_reduced_criterion} is the fact that for every $g \in G$ there is $g_0 \in G$ such that $g \sim_G g_0$ and $g_0$ is cyclically reduced. We will use this fact often without mentioning.

Conjugacy criterion for graph products of groups was proved in \cite[Lemma 3.12]{mf}.
\begin{lemma}[Conjugacy criterion for graph products]
	\label{conjugacy criterion for graph products}
	Let $x, y$ be cyclically reduced elements of $G = \Gamma \mathcal{G}$. Then $x \sim_G y$ if and only if the all of the following are true:
	\begin{enumerate}
		\item[(i)] $|x| = |y|$ and $\supp(x) = \supp(y)$,
		\item[(ii)] $\p(x)$ is a cyclic permutation of $\p(y)$,
		\item[(iii)] $\s(y) \in \s(x)^{G_{\S(x)}}$.
	\end{enumerate}
\end{lemma}

\section{Poitwise inner endomorphisms of graph products}\label{pointwise inner automorphisms}

The aim of this section is to prove Theorem \ref{no central vertex implies property A} and Corollary \ref{theorem A}.

We will need the following technical lemma about conjugators of minimal length.
\begin{lemma}
	\label{PI_technical}
	Let $u \in V\Gamma$, $a \in G_u \setminus \{1\}$ and let $\phi \in \End(G)$. Let $v \in V\Gamma$ and $b \in G_v \setminus \{1\}$. Suppose that $\phi(a) \in G_u \setminus\{1\}$, $\phi(b) \in G_v^G \setminus \{1\}$ and $\phi(ab)\sim_G ab$. Pick $b' \in G_v$ and $w \in G$ such that $\phi(b) = wb'w^{-1}$ and $|w|$ is minimal. Then $w \in N_G(G_u) = G_{\star(u)} = G_{\link(u)}G_u$.
\end{lemma}
\begin{proof}
	By assumption $\phi(a) = a'$, for some $a' \in G_u \setminus \{1\}$. Clearly $w$ can be factorised as reduced product $w = xyz$, where $x \in N_G(G_u)$, $z \in N_G(G_v)$ and $\LL(y) \cap \star(v) = \emptyset = \FL(y) \cap \star(u)$. First we show that $z = 1$. Assume $z \in N_G(G_v) \setminus \{1\}$. Then we can set $b'' = zb'z^{-1}$ and $w' = xy$. Clearly, $\phi(b) = w'b''w'^{-1}$ and $|w'| < |w|$ which is a contradiction with our choice of $b'$ and $w$. We see that $z = 1$.
	
	Now we show that $y = 1$. Obviously 
	\begin{displaymath}
		ab \sim_G \phi(ab) = a' xy b' y^{-1} x^{-1} \sim_G (x^{-1}a'x) y b' y^{-1}.
	\end{displaymath}
	Denote $s = (x^{-1}a'x) y b' y^{-1}$.
	
	Recall that $x^{-1}a'x \in G_u \setminus \{1\}$ and $b' \in G_v \setminus \{1\}$. Suppose that $y \neq 1$. Since $\LL(y) \cap \star(v) = \emptyset = \FL(y) \cap \star(u)$, we see that $s$ is a reduced product of four non-trivial elements of $G$ and thus $|s| \geq 4$. 
	
	Note that $\FL(s) = \{u\}$ and $\LL(s) = \LL(y^{-1}) = \FL(y)$ and thus $s$ is cyclically reduced by Lemma \ref{graph_product_cyclically_reduced_criterion} as $\FL(s) \cap \LL(s) = \emptyset$. Clearly $|s| > 2$, but also $s \sim_G ab$ and $|ab| \leq 2$ which is a contradiction with Lemma \ref{conjugacy criterion for graph products}, thus $y = 1$ and consequently $w \in N_G(G_u)$. By Remark \ref{normaliser} we see that $N_G(G_u) = G_{\star(u)} = G_u G_{\link(u)}$. 
\end{proof}

\begin{corollary}
	\label{PI_stable}
	Let $\phi \in \End(G)$ be such that $\phi(g) \sim_G g$ for every $g \in G$ with $|g| = 1$. Suppose that there is $v \in V\Gamma$ and $a \in G_v \setminus \{1\}$ such that $\phi(a) \in G_v$. Then $\phi(G_v) \subseteq G_v$.
\end{corollary}
\begin{proof}
	Take arbitrary $b \in G_v \setminus \{1\}$. We see that $\phi(b) \sim_G b$ by assumption and thus $\phi(b) \neq 1$. Clearly $|ab| \leq 1$ and thus $\phi(ab) \sim_G ab$ as well. Pick $b'\in G_v \setminus \{1\}$ and $w \in G$ such that $\phi(b) = w b' w^{-1}$ and $|w|$ is minimal. By Lemma \ref{PI_technical} we see that $w \in G_v G_{\link(v)} = N_G(G_v)$ and thus $\phi(b) \in G_v$. We see that $\phi(G_v) \subseteq G_v$. 
\end{proof}

Corollary \ref{PI_stable} tells us that for $v\in V\Gamma$ and $\phi \in \End(G)$, such that $\phi(g) \sim g$ for every $g \in G$ with $|g|=1$, we have either $G_v \cap \phi(G_v) = \{1\}$ or $\phi(G_v)\subseteq G_v$. Therefore it makes sense to give the following definition. Let $G = \Gamma\mathcal{G}$ be a graph product and let $\phi \in \End(G)$, such that $\phi(g) \sim g$ for every $g \in G$ with $|g|=1$. We say that a vertex $v \in V\Gamma$ is \emph{stabilised} by $\phi$ if $\phi(G_v) \subseteq G_v$. We say that a subset $S \subseteq V\Gamma$ is \emph{stabilised} by $\phi$ if every vertex $v \in S$ is stabilised by $\phi$.

Lemma \ref{PI_technical} together with Corollary \ref{PI_stable} allows us to formulate the following corollary as an immediate consequence. Recall that for $A, B \subseteq V\Gamma$ we have $G_A \cap G_B = G_{A \cap B}$.
\begin{corollary}
	\label{PI_technical_cor}
	Let $\phi \in \End(G)$ such that $\phi(g) \sim_G g$ for every $g \in G$ with $|g| \leq 2$. Let $V_0 \subseteq V\Gamma$ be stabilised by $\phi$ and assume that $V_0 \neq \emptyset$. Let $b \in G_v \setminus\{1\}$ be arbitrary, for some $v \in V\Gamma$. Pick $b' \in G_v \setminus \{1\}$ and $w \in G$ such that $\phi(b) = w b' w^{-1}$ and $|w|$ is minimal. Then
	\begin{displaymath}
		w \in \bigcap_{u \in V_0}N_G(G_u) = \bigcap_{u \in V_0}G_{\star(u)} = G_S,
	\end{displaymath}
	where $S = \cap_{v\ \in V_0}\star(v)$.
\end{corollary}

Recall that a vertex $v \in V\Gamma$ is called \emph{central} if $\link(v) = V\Gamma \setminus \{v\}$. Clearly if $v \in V\Gamma$ is a central vertex then $G = G_v \times G_{V\Gamma \setminus \{v\}}$. Note that if $V\Gamma = \{v\}$  then $v$ is central, hence if the graph $\Gamma$ does not contain a central vertex then necessarily $|V\Gamma| \geq 2$.

\begin{lemma}
	\label{almost pointwise inner is inner}
	Let $\Gamma$ be a graph and suppose that there is $U \subseteq V\Gamma$ such that $|U| < \infty$ and $U$ is coneless. Let $\mathcal{G} = \{G_v \mid v \in V\Gamma\}$ be a family of non-trivial groups and let $G = \Gamma \mathcal{G}$ be the graph product of $\mathcal{G}$ with respect to $\Gamma$. Let $\phi_0 \in \End(G)$ and assume that $\phi_0(g) \sim_G g$ for all $g \in G$ such that $|g| \leq 2$. Then $\phi_0 \in \Inn(G)$.
\end{lemma}
\begin{proof}
	Pick $\phi \in \Inn(G) \phi_0$ such that the subset of vertices $V_0 \subseteq U$ stabilised by $\phi$ is maximal. Evidently $V_0 \neq \emptyset$. Denote 
	\begin{displaymath}
		N = \bigcap_{u \in V_0}N_G(G_u) = \bigcap_{u \in V_0}G_{\star(u)} = G_S,
	\end{displaymath}
	where $S = \bigcap_{u \in V_0}\star(u)$. First we show that all vertices of $U$ are stabilised by $\phi$. Suppose  that $V_0 \neq U$. Take $v \in U \setminus V_0$ and let $b \in G_v \setminus \{1\}$ be arbitrary. Pick $b' \in G_v \setminus \{1\}$ and $w \in G$ such that  $\phi(b) = w b' w^{-1}$ and $w \in G$ and $|w|$ is minimal. Note that $b \sim_{G} b'$. By Corollary \ref{PI_technical_cor} we see that $w \in N$. Let $\phi_w$ be the inner automorphism corresponding to $w$. Note that $\phi_w^{-1}\circ \phi \in \Inn(G)\phi = \Inn(G)\phi_0$. Clearly $\phi_w(G_{u})=G_u$ for all $u \in V_0$ and thus $\left(\phi_w^{-1}\circ \phi \right)(G_u) \subseteq G_u$. Also we see that $\left(\phi_w^{-1}\circ \phi\right) (b) = b' \in G_v$ and thus by Corollary \ref{PI_stable} we see that $\left( \phi_w^{-1}\circ \phi \right)(G_v) \subseteq G_v$ which is a contradiction as $\phi$ was chosen so that the subset $V_0$ of stabilised vertices of $U$ is maximal. Hence we see that $V_0 = U$. This implies that $S = \bigcap_{u \in V_0}G_{\star(u)} = \emptyset$ and consequently $G_S = \{1\}$.
	
	 Now we show that all vertices of $\Gamma$ are stabilised. Let $v \in V\Gamma \setminus U$ be arbitrary, take $u \in U$, such that $\{u,v\} \not \in E\Gamma$ (such $u$ always exists as $U$ is coneless) and let $a \in G_u\setminus\{1\}$, $b \in G_v\setminus \{1\}$ be arbitrary. Again, pick $b' \in G_v \setminus \{1\}$ and $w \in G$ such that  $\phi(b) = w b' w^{-1}$ and $|w|$ is minimal. As all the vertices in $U$ are stabilised by $\phi$ we see by Corollary \ref{PI_technical_cor} that $w \in G_S = \{1\}$. We see that $\phi(b) = b'\in G_v$ and thus $\phi(G_v) \subseteq G_v$ by Corollary \ref{PI_stable}. It follows that $\phi(G_v) \subseteq G_v$ for every $v \in V\Gamma$.
	 
	 Finally, we show that $\phi$ is the identity on $G$. Let $v \in V\Gamma$ be given. Again, we can find $u \in U$ such that $\{u,v\} \not\in E\Gamma$. Let $a \in G_u \setminus \{1\}$ and $b \in G_v \setminus \{1\}$ be arbitrary. Clearly $\langle G_u, G_v \rangle = G_{\{u,v\}} \cong G_u \ast G_v$. Since $G_{\{u,v\}}$ is a retract and $\phi(a)\phi(b) \in G_u G_v \subseteq G_{\{u,v\}}$, we see that $ab \sim_{G_{\{u,v\}}} \phi(a)\phi(b)$. By the conjugacy criterion for free products \cite[Theorem 4.2]{magnus} we get that $\phi(a) = a$ and $\phi(b) = b$ thus $\phi\restriction_{G_v} = \id_{G_v}$. We see that $\phi \restriction_{G_v} = \id_{G_v}$ for every $v \in V\Gamma$ and thus $\phi = \id_G$.
	 
	 It follows that $\phi_0 \in \Inn(G)$.
\end{proof}
Note that Lemma \ref{almost pointwise inner is inner} immediately implies Theorem \ref{no central vertex implies property A}.
\begin{proof}[Proof of Theorem \ref{no central vertex implies property A}]
	Let $\phi \in \End(G)$ be arbitrary and suppose that $\phi(g) \sim g$ for all $g \in G$. Then $\phi \in  \Inn(G)$ by Lemma \ref{almost pointwise inner is inner}. In particular, we see that $\Autpi(G) \subseteq \Inn(G)$ and thus $G$ has Grossman's property (A).
\end{proof}

Let us discuss when does a graph $\Gamma$ satisfy the assumptions of Lemma \ref{almost pointwise inner is inner}, i.e. when does a graph $\Gamma$ contain a finite coneless subset? Obviously, if $\Gamma$ is finite then $\Gamma$ contains a coneless subset if and only if $\Gamma$ does not contain a central vertex. For a graph $\Gamma$ we define the \emph{complement graph} $\Gamma^{\mathop{c}}$ in the following way: $V\Gamma^{\mathop{c}} = V\Gamma$ and $E\Gamma^{\mathop{c}} = \binom{V\Gamma}{2} \setminus E \Gamma$. Obviously, the graph $\Gamma$ is irreducible if and only if $\Gamma^{\mathop{c}}$ is connected. Every connected component $\Gamma_1^{\mathop{c}} \leq \Gamma^{\mathop{c}}$ corresponds to a full subgraph $\Gamma_1 \leq \Gamma$, where $V\Gamma_1^{\mathop{c}} = V\Gamma_1$, and we say that $\Gamma_1$ is an \emph{irreducible component} of $\Gamma$. One can easily check that $\Gamma$ contains a finite coneless subset if and only if $\Gamma$ has only finitely many irreducible components and does not contain a central vertex.

We leave the proof of the following lemma as a simple exercise for the reader.
\begin{lemma}
	\label{property A in direct products}
	Let $G_1, \dots, G_n$ be groups. The group $G = \Pi_{i=1}^n G_i$ has Grossman's property (A) if and only if the group $G_i$ has Grossman's property (A) for each $i = 1, \dots,n$.
\end{lemma}
	
Now we are ready to prove Corollary \ref{theorem A}
\begin{proof}[Proof of Corollary \ref{theorem A}]
	Let $C \subseteq V\Gamma$ denote the set of central vertices of the graph $\Gamma$. Note that the induced full subgraph $\Gamma_{V\Gamma \setminus C}$ does not contain central vertices, hence the group $G_{V\Gamma \setminus C}$ has Grossman's property (A) by Theorem \ref{no central vertex implies property A}. The group $G$ splits as $G = G_{\Gamma V \setminus C} \times \prod_{v \in C} G_v$, a direct product of finitely many groups. By Lemma \ref{property A in direct products} we see that the group $G$ has Grossman's property (A) if and only if $G_v$ has Grossman's property (A) for every $v \in C$.	
\end{proof}

Note that Corollary \ref{theorem A} does not hold for infinite graphs. Let $\Gamma$ be a complete graph on countably infinitely many vertices and let $\{G_v \mid v \in V\Gamma\}$ be a family of groups such that $G_v \cong F_2$, where $F_2$ is the free group on two generators, for every $v \in V\Gamma$. We see that $G = \Gamma\mathcal{G}$ is isomorphic to $\prod_{n \in \mathbb{N}}F_2$. Let $w \in F_2 \setminus \{1\}$ be arbitrary and consider the automorphism $\phi_w \in \Aut(G)$ defined on the coordinates as follows: 
\begin{displaymath}
	\phi_w(f_1, f_2, f_3, \dots) = (wf_1 w^{-1}, w^2 f_2 w^{-2}, w^3 f_3 w^{-3}, \dots).
\end{displaymath}
	It is obvious that $\phi_w \in \Autpi(G)\setminus \Inn(G)$, hence $G$ does not have Grossman's property (A), but $F_2$ has Grossman's property (A) by \cite[Lemma 1]{grossman}. Note that the group $G$ can be actually obtained as a RAAG corresponding to an infinite graph without central vertices. However, this graph does not contain a finite coneless subset.

In the rest of the section we prove three technical results about conjugacy in graph products of groups that will be useful in Section \ref{CD pairs in GP}
\begin{lemma}
	\label{direct}
	Let $u,v \in V\Gamma$ be such that $\{u,v\} \in E\Gamma$ and let $a \in G_u \setminus\{1\}$, $b \in G_v \setminus \{1\}$ be arbitrary. Let $\phi_0 \in \End(G)$ and assume that $\phi_0(a) \sim_G a$ and $\phi_0(b) \sim_G b$. Then $\phi_0(ab) \sim_G ab$. 
\end{lemma}
\begin{proof}
	Pick $\phi \in \Inn(G)\phi_0$ such that $\phi(a) = a$. Then $\phi(b) = c b c^{-1}$ for some $c \in G$. We see that $\phi(a)\phi(b) = \phi(b) \phi(a)$ and thus $cbc^{-1} \in C_G(a)$. By Lemma \ref{centralisers_in_graph_products} we see that $C_G(a) = C_{G_u}(a)G_{\link(u)} \leq G_u G_{\link(u)}$. Note that $G_u G_{\link(u)} = G_{\{u\}\cup\link(u)}$ is a retract of $G$, let $\rho \colon G \to G_u G_{\link(u)}$ be the corresponding retraction. Since $v \in \link(u)$ we see that $\rho(b) = b$. We see that $cbc^{-1} =\rho(cbc^{-1}) = \rho(c)b\rho(c)^{-1}$. Set $c_1 = \rho(c)$. Since $c_1^{-1} \phi(ab) c \in G_{\{u\}\cup\link(v)}$, we see that $c_1^{-1}\phi(ab)c_1 = \rho(c_1^{-1} \phi(ab) c_1) = c_1^{-1} a c_1 b$. Since $c_1 \in G_u G_{\link(u)}$ there is $c_2 \in G_u$ such that $c_1^{-1} a c_1 = c_2^{-1} a c_2$. As $c_2 \in G_u$ and $\{u,v\} \in E\Gamma$ we see that 
	\begin{displaymath}
		c_1^{-1}\phi(ab)c_1 = c_1^{-1} a c_1 b =c_2^{-1} a c_2b = c_2^{-1} a b c_2.
	\end{displaymath}
	It follows that $\phi_0(ab) \sim_G ab$.
\end{proof}

\begin{corollary}
		\label{direct product}
Let $\phi \in \End(G)$. Assume that $\phi(g) \sim_G g$ for every $g \in G$ such that $g$ is cyclically reduced and $\S(g) = \emptyset$. Furthermore, suppose that $\phi(g) \sim_G g$ for all $g \in G$ such that $|g| = 1$ as well. Then $\phi(g) \sim_G g$ for all $g \in G$ such that $|g| \leq 2$. 
\end{corollary}
\begin{proof}
	Let $g \in G$ be arbitrary such that $|g| = 2$. Clearly, $\supp(g) = \{u,v\}$ for some $u,v \in V\Gamma$ such that $u \neq v$. One can easily check that $g$ is cyclically reduced using Lemma \ref{graph_product_cyclically_reduced_criterion}. Suppose that $\{u,v\} \not\in E\Gamma$. Then $\S(g) = \emptyset$ and $\phi(g) \sim_G g$ by assumption.
	
	Now suppose that $\{u,v\} \in E\Gamma$. Then $g = ab$ for some $a \in G_u\setminus\{1\}$, $b \in G_v \setminus \{1\}$. By assumption, $\phi(a) \sim_G a$ and $\phi(b) \sim_G b$ as $|a| = |b| =1$. Then $\phi(ab) \sim_G ab$ by the previous lemma and we are done.
\end{proof}

\begin{lemma}
	\label{almost pointwise inner}
	Let $\phi_0 \in \End(G)$ and let $u,v \in V\Gamma$ be such that $\{u,v\} \not\in E\Gamma$ and $u \neq v$. Let $a \in G_{u}\setminus \{1\}$ and $b \in G_v \setminus \{1\}$ be arbitrary and assume that $\phi_0(ab) \sim_G ab$, $\phi_0(a) \in G_u^G \setminus\{1\}$ and $\phi_0(b) \in G_v^G \setminus \{1\}$. Then $\phi_0(a) \sim_G a$ and $\phi_0(b) \sim_G b$.
\end{lemma}
\begin{proof}
	By assumption $\phi_0(a) = w_a a' w_a^{-1}$ for some $a' \in G_u \setminus \{1\}$. Set $\phi = \phi_{w_a}^{-1} \circ \phi_0$, where $\phi_{w_a} \in \Inn(G)$ is the inner automorphism of $G$ corresponding to $w_a$. Clearly $\phi(ab) \sim_G ab$, $\phi(a) = a' \in G_u \setminus\{1\}$ and $\phi(b) \in G_v^G \setminus\{1\}$. Pick $b' \in G_v \setminus \{1\}$ and $w \in G$ such that $\phi(b) = wb'w^{-1}$ and $|w|$ is minimal. By Lemma \ref{PI_technical} we see that $w \in G_{\link(u)}G_u = N_G(G_u)$. We have $ab \sim_G \phi(ab) = a' wb'w^{-1}$ and consequently $ab \sim_G w^{-1} a' w b'$. Note that $w^{-1} a' w \in G_u$ since $w \in N_G(G_u)$. Denote $a'' = w^{-1} a' w$. Let $\rho \colon G \to G_{\{u,v\}}$ be the canonical retraction corresponding to the set of vertices $\{u,v\}$. Clearly $\rho(ab) = ab$ and $\rho(a''b') = a''b'$ and $ab \sim_{G_{\{u,v\}}} a'' b'$. Note that $G_{\{u,v\}} \cong G_u \ast G_v$ and thus by the conjugacy criterion for free products of groups (see \cite[Theorem 4.2]{magnus}) we see that $a'' = a$ and $b' = b$. It follows that $\phi_0(a) \sim_G a$ and $\phi_0(b) \sim b$.	
\end{proof}

\section{Conjugacy distinguishable pairs in graph products}\label{CD pairs in GP}
We say that a class $\C$ is an \emph{extension closed variety of finite groups} if the class $\C$ of finite groups is closed under taking subgroups, finite direct products, quotients and extensions. Obvious examples of extension closed varieties of finite groups are the following:
\begin{itemize}
	\item the class of all finite groups;
	\item the class of all finite $p$-groups, where $p$ is a prime number;
	\item the class of all finite solvable groups.
\end{itemize}
Unless stated otherwise (see Lemma \ref{separating map}, Lemma \ref{separable pairs} and Lemma \ref{empty stem}), in this section we will assume that the class $\C$ is an extension closed variety of finite groups. This will allow us to use the following lemma which is a direct consequence of \cite[Theorem 1.2]{mf} 
\begin{lemma}
	\label{finite graph products of finite groups are CS}
	Let $\C$ be an extension closed variety of finite groups and let $G = \Gamma \mathcal{G}$ be a graph product of $\C$-groups. Then the group $G$ is $\C$-CS.
\end{lemma}

The main result of this section is the following proposition.
\begin{proposition}
	\label{no central vertices implies CIS}
	Let $\Gamma$ be a finite simplicial graph without central vertices and let $\mathcal{G} =\{G_v \mid v \in V\Gamma\}$ be a family of non-trivial finitely generated residually-$\C$ groups. Then the group $\Gamma \mathcal{G}$ is $\C$-IAS.
\end{proposition}
To prove Proposition \ref{no central vertices implies CIS} we will give sufficient conditions for the pair $(f,g)$ to be $\C$-CD in the graph product (see Lemma \ref{CD pairs}) and then use this description to show that if we have and automorphism $\phi$ such that $(g,\phi(g))$ is not $\C$-CD for all $g \in G$ then necessarily $\phi$ must be inner.

The following remark demonstrates another useful property of graph products: the functorial property, i.e. that homomorphisms between vertex groups uniquely extend to a homomorphisms of graph products. 
\begin{remark}
\label{functional property of graph products}
Let $\Gamma$ be a simplicial graph and let $\mathcal{G} = \{G_v \mid v \in V\Gamma\}$, $\mathcal{F} = \{F_v \mid v \in V\Gamma\}$ be two families of groups such that for every $v \in V\Gamma$ there is a homomorphism $\phi_v \colon G_v \to F_v$. Then there is unique group homomorphism $\phi \colon \Gamma\mathcal{G} \to \Gamma\mathcal{F}$ such that $\phi \restriction_{G_v} = \phi_v$ for every $v \in V\Gamma$.
\end{remark}

The following lemma is an easy consequence of \cite[Lemma 7.2]{mf}.
\begin{lemma}
\label{separating map}
Suppose that C is a class of finite groups closed under taking direct products and subgroups. Let $\Gamma\mathcal{G}$ be a graph product of residually-$\C$ groups. Let $f, g \in G$ be cyclically reduced in $G$ and assume that $f \neq g$. Then there is $\mathcal{F} = \{F_v | v\in V\Gamma\}$, a family of $\C$-groups indexed by $V\Gamma$, and a homomorphism $\phi_v \colon G_v \to F_v$, for every $v \in V\Gamma$, such that for the corresponding extension $\phi \colon G \to F$, where $F = \Gamma\mathcal{F}$, all of the following are true:
	\begin{itemize}
		\item[(i)] $|g| = |\phi(g)|$ and $\supp(g) = \supp(\phi(g))$,
		\item[(ii)] $|f| = |\phi(f)|$ and $\supp(f) = \supp(\phi(f))$, 
		\item[(iii)] $\phi(f), \phi(g)$ are cyclically reduced in $F$,
		\item[(iv)] $\phi(f) \neq \phi(g)$ in $F$.
	\end{itemize}
\end{lemma}

We utilise Lemma \ref{separating map} to show that conjugacy classes of certain pairs of elements of graph products of residually-$\C$ groups can be separated in a graph product of $\C$-groups.
\begin{lemma}
	\label{separable pairs}
	Suppose that $\C$ is a class of finite groups closed under taking direct products and subgroups. Let $\Gamma$ be a graph and let $\mathcal{G} = \{G_v \mid v \in V\Gamma\}$ be a family of residually-$\C$ groups. Let $G = \Gamma\mathcal{G}$ and suppose that $f,g \in G$ are cyclically reduced elements of $G$ such that $f \not\sim_G g$ and either $\supp(f) \neq \supp(g)$ or $|f|\neq|g|$. Then there is a family of $\C$-groups $\mathcal{F} = \{F_v \mid v\in V\Gamma\}$ and a homomorphism $\phi \colon G \to F$, where $F = \Gamma\mathcal{F}$, such that $\phi(f) \not\sim_F \phi(g)$.
\end{lemma}
\begin{proof}
	By Lemma \ref{separating map} there is a family $\C$-groups $\mathcal{F} = \{F_v \mid v \in V\Gamma\}$ such that for every $v \in V\Gamma$ there is a homomorphism $\phi_v \colon G_v \to F_v$, such that for the corresponding extension $\phi \colon G \to F = \Gamma \mathcal{F}$ we have $|g| = |\phi(g)|$, $\supp(g) = \supp(\phi(f))$, $|f| = |\phi(f)|$, $\supp(f) = \supp(\phi(f))$ and both $\phi(f)$, $\phi(g)$ are cyclically reduced. By Lemma \ref{conjugacy criterion for graph products} we see that $\phi(f) \not\sim_F \phi(g)$.
\end{proof}

As it turns out, conjugacy classes of cyclically reduced elements with specific P-S decomposition can be always separated.
\begin{lemma}
	\label{empty stem}
	Suppose that $\C$ is a class of finite groups closed under taking direct products and subgroups. Let $\Gamma$ be a graph and let $\mathcal{G} = \{G_v \mid v \in V\Gamma\}$ be a family of residually-$\C$ groups. Let $G = \Gamma\mathcal{G}$ be a graph product of $\mathcal{G}$ with respect to $\Gamma$. Let $f\in G$ be cyclically reduced element of $G$ such that $\S(f) = \emptyset$. Then for every $g \in G \setminus f^G$ there is a family of $\C$-groups $\mathcal{F} = \{F_v \mid v \in V\Gamma\}$ and a homomorphism $\phi \colon G \to F$, where $F = \Gamma\mathcal{F}$, such that $\phi(f)\not\sim_F \phi(g)$.
\end{lemma}
\begin{proof}
	Let $g \in G \setminus f^G$ be arbitrary. Pick $g_0 \in G$ such that $g_0 \sim_G g$ and $g_0$ is cyclically reduced. Since $\S(f) = \emptyset$ we see that $\p(f) = f$. Combining Lemma \ref{conjugacy criterion for graph products} with the fact that $\S(f) = \emptyset$ we see that there are two possibilities to consider:
	\begin{itemize}
		\item[(i)] $\supp(f) \neq \supp(g_0)$ or $|f| \neq |g_0|$,
		\item[(ii)] $\p(g_0)$ is not a cyclic permutation of $f$.
	\end{itemize}
	
	We can use Lemma \ref{separable pairs} do deal with case (i).
	
	Assume that $\supp(f) = \supp(g_0)$, $|f| = |g_0|$ and that $\p(g_0)$ is not a cyclic permutation of $f$. Let $\{f_1, \dots, f_m\} \subset G$ be the set of all cyclic permutations of $f$ (including $f$). We use Lemma \ref{separating map} for each pair $f_i, g_0$, where $1 \leq i \leq m$, to obtain a family $\C$-groups $\mathcal{F}_i = \{F_v^i | v \in V\Gamma\}$ with homomorphisms $\phi_v^i \colon G_v \to F_v^i$ for all $v \in V\Gamma$. For every $v \in V\Gamma$ set $K_v = \bigcap_{i = 1}^m \ker(\phi_v^i)$ and denote $F_v = G_v / K_v$. Note that as the class $\C$ is closed under taking subgroups and direct products the set $\NC(G_v)$ is closed under intersection for every $v \in V\Gamma$ (see \cite[Lemma 2.1]{mf}) and thus $F_v \in \C$ for every $v \in V\Gamma$. Set $\mathcal{F} = \{F_v | v\in V\Gamma\}$ and let $\phi_v \colon G_v \to F_v$ be the natural projection corresponding to $v$. Let $\phi \colon G \to \Gamma \mathcal{F}$ be the natural extension. Note that $\p(\phi(g_0)) = \phi(\p(g_0))$, $\p(\phi(f)) = \phi(\p(f))$ and $\phi(f)$, $\phi(g_0)$ are cyclically reduced in $\Gamma\mathcal{F}$. Clearly the set $C = \{\phi(f_1),\dots, \phi(f_m)\}$ is the set of all cyclic permutations of $\p(\phi(f))$ and we see that $\p(\phi(g_0))\not\in C$, hence $\p(\phi(g_0))$ is not a cyclic permutation of $\phi(f)$. By Lemma \ref{conjugacy criterion for graph products} we see that $\phi(g_0) \not\sim_{\Gamma\mathcal{F}} \phi(f)$ and thus $\phi(g) \not\sim_{\Gamma\mathcal{F}} \phi(f)$.
\end{proof}

Combining Lemma \ref{separable pairs} and Lemma \ref{empty stem} together with Lemma \ref{finite graph products of finite groups are CS} we get the following description of $\C$-CD pairs in graph products.
\begin{lemma}
	\label{CD pairs}
	Let $\Gamma$ be a graph and let $\mathcal{G} = \{G_v \mid v \in V\Gamma\}$ be a family of residually-$\C$ groups. Let $G = \Gamma\mathcal{G}$ be a graph product of $\mathcal{G}$ with respect to $\Gamma$. Let $g_1,g_2 \in G$ be cyclically reduced elements of $G$ such that $g_1 \not\sim_G g_2$ and either $\supp(g_1) \neq \supp(g_2)$ or $|g_1|\neq|g_2|$. Then the pair $(g_1,g_2)$ is $\C$-CD in $G$. Furthermore, if $f \in G$ is cyclically reduced with $\S(f) = \emptyset$ then the pair $(f,g)$ is $\C$-CD for every $g \in G \setminus f^G$.
\end{lemma}
\begin{proof}
	If $g_1, g_2 \in G$ are cyclically reduced and either $\supp(g_1) \neq \supp(g_2)$ or $|g_1| \neq |g_2|$ then by Lemma \ref{separable pairs} we see that there is a family of $\C$-groups $\mathcal{F} = \{F_v \mid v \in V\Gamma\}$ and a homomorphism $\gamma \colon G \to F = \Gamma\mathcal{F}$ such that $\gamma(g_1) \not\sim_F \gamma(g_2)$. By Lemma \ref{finite graph products of finite groups are CS} we see that the group $F$ is $\C$-CS and thus the pair $(\gamma(g_1),\gamma(g_2))$ is $\C$-CD in $F$. Using Remark \ref{CCD simplification} we see that the pair $(f,g)$ is $\C$-CD in $G$.
	
	Similarly, if $f \in G$ is cyclically reduced with $\S(f) = \emptyset$ then by Lemma \ref{empty stem} there is a family $\mathcal{F} = \{F_v \mid v \in V\Gamma\}$ and a homomorphism $\gamma \colon G \to F = \Gamma \mathcal{F}$ such that $\gamma(f) \not\sim_F \gamma(g)$. Again, by Lemma \ref{finite graph products of finite groups are CS} we see that the group $F$ is $\C$-CS and thus the pair $(\gamma(f),\gamma(g))$ is $\C$-CD in $F$. Using Remark \ref{CCD simplification} we see that the pair $(g_1, g_2)$ is $\C$-CD in $G$.
\end{proof}

\begin{lemma}
Let $\Gamma$ be a graph and suppose that there is $U \subseteq V\Gamma$ such that $|U| < \infty$ and $U$ is coneless. Let $\mathcal{G} = \{G_v \mid v \in V\Gamma\}$ be a family of non-trivial residually-$\C$ groups and let $G = \Gamma \mathcal{G}$ be the graph product of $\mathcal{G}$ with respect to $\Gamma$. Then for every $\phi \in \Aut(G) \setminus \Inn(G)$ there exists $g \in G$ such that $\phi(g)\not\sim_G g$ and the pair $(\phi(g),g)$ is $\C$-CD in $G$.
\end{lemma}
\begin{proof}
	Let $\phi \in \Aut(G) \setminus \Inn(G)$ be arbitrary and assume that for every $g \in G$ the pair $(\phi(g),g)$ is not $\C$-CD in $G$.
		
	If $g \in G$ is cyclically reduced with $\S(g) = \emptyset$ then the pair $(f,g)$ is $\C$-CD for every $f \in G \setminus g^G$ by Lemma \ref{CD pairs}. We see that we may assume that $\phi(g) \sim_G g$ for every cyclically reduced element $g \in G$ such that $\S(g) = \emptyset$. In particular $\phi(ab)\sim_G ab$, whenever $a \in G_u \setminus \{1\}$ and $b \in G_v \setminus \{1\}$ for some $u,v \in V\Gamma$ such that $\{u,v\} \not\in E\Gamma$. 
	
	Let us analyse what happens to $g \in G$ with $|g| = 1$. Let $u \in V\Gamma$ and $a \in G_u \setminus \{1\}$ be arbitrary. Pick $h \in G$ such that $h \sim_G \phi(a)$ and $h$ is cyclically reduced. There are three cases to consider:
	\begin{itemize}
		\item[(i)]		$1<|h|$,
		\item[(ii)]		$|h| = 1$ and $\supp(h) \neq \{u\} = \supp(a)$,
		\item[(iii)]	$|h| = 1$ and $\supp(h) = \{u\}$.
	\end{itemize} 
	Using Lemma \ref{CD pairs} we see that if (i) or (ii) is the case then the pair $(a,h)$ is $\C$-CD. This means that there is a group $C \in \C$ and a homomorphism $\gamma \colon G \to C$ such that $\gamma(a) \not\sim_C \gamma(h)$. Consequently $\gamma(\phi(a)) \not\sim_C \gamma(a)$ and the pair $(\phi(a),a)$ is $\C$-CD in $G$. We see that without loss of generality we may assume that $\phi(g) \in G_v^G$, whenever $g \in G_v \setminus \{1\}$ for some $v \in V\Gamma$, because otherwise the pair $(\phi(g),g)$ would be conjugacy distinguishable as we just demonstrated.
	
	As $\Gamma$ does not contain central vertices we know that for every $u \in V\Gamma$ there is $v \in V\Gamma \setminus \{u\}$ such that $\{u,v\} \not \in E\Gamma$. Let $b \in G_v\setminus \{1\}$ be arbitrary. We see that $\phi(a) \in G_u^G$ and $\phi(b) \in G_v^G$. Clearly the element $ab$ is cyclically reduced and $\S(ab) = \emptyset$, hence $\phi(ab) \sim_G ab$ by assumption. Then by Lemma \ref{almost pointwise inner} we see that $\phi(a) \sim_G a$ and $\phi(b) \sim_G b$. This means that we may assume that $\phi(g) \sim g$ for all $g \in G$ such that $|g|=1$. Consequently, by Corollary \ref{direct product} we see that $\phi(g) \sim_G g$ for all $g \in G$ such that $|g| \leq 2$. However, using Lemma \ref{almost pointwise inner is inner} we see that $\phi \in \Inn(G)$, which is a contradiction with our original assumption that $\phi \in \Aut(G) \setminus \Inn(G)$.
	
	We see that our original assumption cannot be true, i.e. there must be an element $g \in G$ such that $\phi(g)\not\sim_G g$ and the pair $(\phi(g),g)$ is $\C$-CD in $G$.
\end{proof}

Now we are ready to prove Proposition \ref{no central vertices implies CIS}
\begin{proof}[Proof of proposition \ref{no central vertices implies CIS}]
By previous lemma we see that for every $\phi \in \Aut(G) \setminus \Inn(G)$ there is $g \in G$ such that $\phi(g)\not\sim_G g$ and the pair $(\phi(g),g)$ is $\C$-CD in $G$. As $G$ is a finite graph product of finitely generated groups it is finitely generated and thus by Lemma \ref{my Grossman} we see that the group $G$ is $\C$-IAS.
\end{proof}

\begin{corollary}
	\label{CIS graph products}
	Let $\Gamma$ be a finite graph and let $\mathcal{G} = \{G_v \mid v \in V\Gamma\}$ be a family of non-trivial finitely generated residually-$\C$ groups such that the group $G_v$ is $\C$-IAS whenever the vertex $v$ is central in $\Gamma$. Then the group $G = \Gamma\mathcal{G}$ is $\C$-IAS.
\end{corollary}
\begin{proof}
	Let $C \subseteq V\Gamma$ denote the set of central vertices of graph $\Gamma$. Note that the induced full subgraph $\Gamma_{V\Gamma \setminus C}$ does not contain central vertices, hence the group $G_{V\Gamma \setminus C}$ is $\C$-IAS by Lemma \ref{no central vertices implies CIS}. The group $G_{V\Gamma \setminus C}$ is residually-$\C$ by \cite[Lemma 6.6]{mf}. The group $G$ splits as $G = G_{\Gamma V \setminus C} \times \prod_{v \in C} G_v$, a direct product of finitely many finitely generated $\C$-IAS residually-$\C$ groups, and thus $G$ is $\C$-IAS by Proposition \ref{cias-direct}.	
\end{proof}

Applying Proposition \ref{no central vertices implies CIS} and Corollary \ref{CIS graph products} to the class of all finite groups we immediately obtain Theorem \ref{corollary O} and Corollary \ref{theorem O}.

\section{Graph products of residually-$p$ groups}
\label{section-p}
Let $G$ be a group and let $p$ be a prime number. Set $K_p = [G,G]G_p \leq G$, where $G_p$ is the subgroup of $G$ generated by all elements of the form $g^p$ for $g \in G$. Note that $K_p$ is characteristic in $G$ and thus the natural projection $\pi \colon G \to G/K_p$ induces a homomorphism $\tilde{\pi} \colon \Aut(G) \to \Aut(G/K_p)$ given by $\tilde{\pi}(\phi)(gK_p) = \phi(g)K_p$ for every $\phi \in \Aut(G)$. We will use $\Aut_p(G)$ to denote $\ker(\tilde{\pi})$, i.e. the automorphisms that act trivially on the first mod-$p$ homology of $G$. Note that if $G$ is finitely generated then $G/K_p$ is actually the direct product of copies of $C_p$, the cyclic group of order $p$, and we see that $G/K_p$ is a finite $p$-group and thus $K_p$ is of finite index in $G$. Consequently, if $G$ is finitely generated then $\Aut_p(G)$ is of finite index in $\Aut(G)$. Also since $G/K_p$ is abelian we see that $\Inn(G) \leq \Aut_p(G)$ and thus $\Out_p(G) = \Aut_p(G)/\Inn(G) \leq \Out(G)$. Again, if $G$ is finitely generated then $\Out_p(G)$ is actually of finite index in $\Out(G)$.

The following is a classical result of P. Hall (see \cite[5.3.2, 5.3.3]{robinson}).
\begin{lemma}
	\label{halls lemma}
	If $G$ is a finite $p$-group, then $\Aut_p(G)$ is also a finite $p$-group.
\end{lemma}

Recall that if $\C$ is the class of all finite $p$-groups, then the corresponding pro-$\C$ topology on a group $G$ is referred to as the pro-$p$ topology on $G$. We say that a subset $X \subseteq G$ is $p$-closed in $G$ if it is closed in pro-$p(G)$. If group $G$ is $\C$-IAS then we say that $G$ is $p$-IAS, similarly for $p$-CS and $p$-Grossmanian groups. 

\begin{lemma}
	\label{p-CIAS virtually residually p}
	Let $G$ be a finitely generated $p$-IAS group. Then the group $\Out_p(G)$ is residually $p$-finite and, consequently, the group $\Out(G)$ is virtually residually $p$-finite.
\end{lemma}
\begin{proof}
	Let $\phi \in \Aut_p(G) \setminus \Inn(G)$ be arbitrary. By definition there is $N \in \mathcal{N}_{p}(G)$ characteristic in $G$ such that the natural projection $\pi \colon G \to G/N$ induces a homomorphism $\tilde{\pi} \colon \Aut(G) \to \Aut(G/N)$ such that $\tilde{\pi}(\phi) \not\in \Inn(G/N)$. By Lemma \ref{halls lemma} we see that $\Aut_p(G/N)$ is a finite $p$-group. Note that $\tilde{\pi}(\Aut_{p}(G)) \leq \Aut_p(G/N)$ and thus $\tilde{\pi}(\phi) \in \Aut_p(G/N) \setminus \Inn(G/N)$, therefore $\Inn(G)$ is $p$-closed in $\Aut_p(G)$ and consequently $\Out_p(G)$ is residually $p$-finite. As $G$ is finitely generated, $\Out_p(G)$ is of finite index in $\Out(G)$ and we see that $\Out(G)$ is virtually residually $p$-finite.
\end{proof}

This gives us everything we need to prove Theorem \ref{virtually residually p}.
\begin{proof}[Proof of Theorem \ref{virtually residually p}]
	Using Theorem \ref{CIS graph products} in the context of the class of all $p$-finite groups we see that the group $\Gamma\mathcal{G}$ is $p$-IAS. The rest follows by Lemma \ref{p-CIAS virtually residually p}. 
\end{proof}

Applying Proposition \ref{cias-direct} to the class of all $p$-finite groups we get the following $p$-analogue of Corollary \ref{ias direct}.
\begin{lemma}
\label{p-IAS direct}
	Let $A,B$ be finitely generated residually $p$-finite $p$-IAS groups. Then $A \times B$ is $p$-IAS and, consequently, $\Out_p(G)$ is residually $p$-finite and $\Out(G)$ is virtually residually $p$-finite.
\end{lemma}
\begin{proof}
	Applying Proposition \ref{cias-direct} to the case when $\C$ is the class of all finite $p$-groups we see that $A \times B$ is $p$-IAS. The rest follows by Lemma \ref{p-CIAS virtually residually p}.
\end{proof}

\begin{proof}[Proof of Corollary \ref{p-corollary}]
	Denote $G = \Gamma \mathcal{G}$. Let $C \subseteq V \Gamma$ be set of central vertices of $\Gamma$. Note that the induced subgraph $\Gamma_{V\Gamma \setminus C}$ does not contain central vertices and thus by Theorem \ref{virtually residually p} we see that the full subgroup $G_{V\Gamma \setminus C}$ is $p$-IAS. We see that $G$ splits as $G = G_{V \setminus C} \times \prod_{v \in C}G_V$, a direct product of $p$-IAS groups. The rest follows by Lemma \ref{p-IAS direct}
\end{proof}

Let $G$ be a group. Consider the natural homomorphism $\pi \colon G \to G/[G,G]$. Clearly $[G,G]$ is a characteristic subgroup of $G$ and thus $\pi$ induces a homomorphism $$\tilde{\pi} \colon \Aut(G) \to \Aut(G/[G,G]).$$ Note that $\Inn(G) \leq \ker(\tilde{\pi})$,  hence $\tilde{\pi}$ induces a homomorphism $$\pi^* \colon \Out(G) \to \Out(G/[G,G]).$$ The kernel of this homomorphism is the \emph{Torreli group} of $G$ ($\Tor(G)$), i.e. $\phi \in \Out(G)$  belongs to $\Tor(G)$ if and only if $\phi$ acts trivially on the Abelianisation of $G$. Note that $\Tor(G) \subseteq \Out_p(G)$ for every prime number $p$.

We say that a group $G$ is bi-orderable if there exist a total ordering $\preceq$ of $G$ such that if $f \preceq g$ then $cf \preceq cg$ and $fc \preceq gc$ for all $c,f,g \in G$.

\begin{proof}[Proof of Theorem \ref{torreli}]
	Every residually torsion-free nilpotent group is residually $p$-finite for every prime $p$ by \cite[Theorem 2.1]{gruenberg} and thus we see that $\Out_p(G)$ is residually $p$-finite by Theorem \ref{virtually residually p}. As discussed earlier, $\Tor(G) \leq \cap_{p\in \mathbb{P}}\Out_p(G)$, where $\mathbb{P}$ denotes the set of all prime numbers. Being residually $p$-finite is a hereditary property and thus $\Tor(G)$ is residually $p$-finite for every prime number $p$. Consequently, by \cite{rhemtulla} we see that $\Tor(G)$ is bi-orderable.
\end{proof}

\section{Open questions}
Let $A,B$ be finitely generated RF groups such that $\Out(A)$ and $\Out(B)$ are RF as well. As follows from Corollary \ref{ias direct}, if we assume that groups $A,B$ are IAS then $\Out(A \times B)$ is RF as well. However, what if we drop this assumption? What can be said about residual finiteness of $\Out(A \times B)$?
\begin{question}
\label{question 1}
	Let $A,B$ be finitely generated RF groups such that $\Out(A), \Out(B)$ are RF. Is $\Out(A\times B)$ RF?
\end{question}

Clearly, if every finitely generated RF group with $\Out(G)$ RF was IAS then the class of finitely generated groups with RF outer automorphism would be closed under taking direct products by Proposition \ref{cias-direct}. This naturally leads to another question. 
\begin{question}
\label{question 2}
	Is there a finitely generated RF group $G$ such that $\Out(G)$ is an infinite RF group but $G$ is not IAS?
\end{question}

\section*{Acknowledgements}
I would like to express my gratitude to Ashot Minasyan and Pavel Zalesskii for many insightful discussions. Also, I would like to thank the anonymous referee for pointing out that Theorem \ref{no central vertex implies property A} can be easily generalised to much broader class of graphs and many more useful suggestions.

\end{document}